\documentclass[12pt,reqno]{amsart}

\usepackage{color}
\usepackage{amsmath, amsthm, amssymb}
\usepackage{amsfonts}
\usepackage[ansinew]{inputenc}
\usepackage[dvips]{epsfig}
\usepackage{graphicx}
\usepackage[english]{babel}
\usepackage{hyperref}
\theoremstyle{plain}
\newtheorem{thm}{Theorem}[section]
\newtheorem{cor}[thm]{Corollary}
\newtheorem{lem}[thm]{Lemma}
\newtheorem{prop}[thm]{Proposition}

\theoremstyle{definition}
\newtheorem{defi}[thm]{Definition}

\theoremstyle{remark}
\newtheorem{rem}[thm]{Remark}

\numberwithin{equation}{section}

\newcommand{\average}{{\mathchoice {\kern1ex\vcenter{\hrule height.4pt
width 6pt depth0pt} \kern-9.7pt} {\kern1ex\vcenter{\hrule
height.4pt width 4.3pt depth0pt} \kern-7pt} {} {} }}

\def\R{\mathbb{R}}

\begin{document}

\title[Boundary regularity for nonlocal equations in $C^1$ and $C^{1,\alpha}$ domains]{Boundary regularity estimates for \\ nonlocal elliptic equations in $C^1$ and $C^{1,\alpha}$ domains}

\author{Xavier Ros-Oton}
\address{The University of Texas at Austin, Department of Mathematics, 2515 Speedway, Austin, TX 78751, USA}
\email{ros.oton@math.utexas.edu}

\author{Joaquim Serra}
\address{Universitat Polit\`ecnica de Catalunya, Departament de Matem\`atiques, Diagonal 647, 08028 Barcelona, Spain}
\email{joaquim.serra@upc.edu}


\maketitle

\begin{abstract}
We establish sharp boundary regularity estimates in $C^1$ and $C^{1,\alpha}$ domains for nonlocal problems of the form $Lu=f$ in $\Omega$, $u=0$ in $\Omega^c$.
Here, $L$ is a nonlocal elliptic operator of order $2s$, with $s\in(0,1)$.

First, in $C^{1,\alpha}$ domains we show that all solutions $u$ are $C^s$ up to the boundary and that $u/d^s\in C^\alpha(\overline\Omega)$, where $d$ is the distance to $\partial\Omega$.

In $C^1$ domains, solutions are in general not comparable to $d^s$, and we prove a boundary Harnack principle in such domains.
Namely, we show that if $u_1$ and $u_2$ are positive solutions, then $u_1/u_2$ is bounded and H\"older continuous up to the boundary.

Finally, we establish analogous results for nonlocal equations with bounded measurable coefficients in non-divergence form.
All these regularity results will be essential tools in a forthcoming work on free boundary problems for nonlocal elliptic operators \cite{CRS-obstacle}.
\end{abstract}

\vspace{4mm}

\section{Introduction and results}

In this paper we study the boundary regularity of solutions to nonlocal elliptic equations in $C^1$ and $C^{1,\alpha}$ domains.
The operators we consider are of the form
\begin{equation}\label{operator-L}
Lu(x)=\int_{\R^n}\left(\frac{u(x+y)+u(x-y)}{2}-u(x)\right)\frac{a(y/|y|)}{|y|^{n+2s}}\,dy,
\end{equation}
with
\begin{equation}\label{ellipt-const}
\qquad \qquad \qquad 0<\lambda\leq a(\theta)\leq \Lambda,\qquad \theta\in S^{n-1}.
\end{equation}
When $a\equiv ctt$, then $L$ is a multiple of the fractional Laplacian $-(-\Delta)^s$.

We consider solutions $u\in L^\infty(\R^n)$ to
\begin{equation}\label{pb}
\left\{ \begin{array}{rcll}
L u &=&f&\textrm{in }B_1\cap \Omega \\
u&=&0&\textrm{in }B_1\setminus\Omega,
\end{array}\right.\end{equation}
with $f\in L^\infty(\Omega\cap B_1)$ and $0\in\partial\Omega$.

\vspace{2mm}

When $L$ is the Laplacian $\Delta$, then the following are well known results:
\begin{itemize}
\item[(i)] If $\Omega$ is $C^{1,\alpha}$, then $u\in C^{1,\alpha}(\overline\Omega\cap B_{1/2})$.
\item[(ii)] If $\Omega$ is $C^1$, then solutions are in general \emph{not} $C^{0,1}$.
\end{itemize}
Still, in $C^1$ domains one has the following boundary Harnack principle:
\begin{itemize}
\item[(iii)] If $\Omega$ is $C^1$, and $u_1$ and $u_2$ are positive in $\Omega$, with $f\equiv0$, then $u_1$ and $u_2$ are comparable in $\overline\Omega\cap B_{1/2}$, and $u_1/u_2\in C^{0,\gamma}(\overline\Omega\cap B_{1/2})$ for some small $\gamma>0$.
\end{itemize}
Actually, (iii) holds in general Lipschitz domains (for $\gamma$ small enough), or even in less regular domains; see \cite{BHP1,BHP2}.
Analogous results hold for more general second order operators in non-divergence form $L=\sum_{i,j}a_{ij}(x)\partial_{ij}u$ with bounded measurable coefficients $a_{ij}(x)$ \cite{BHP3}.

\vspace{2mm}

The aim of the present paper is to establish analogous results to (i) and (iii) for nonlocal elliptic operators $L$ of the form \eqref{operator-L}-\eqref{ellipt-const}, and also for non-divergence operators with bounded measurable coefficients.

\addtocontents{toc}{\protect\setcounter{tocdepth}{1}}  
\subsection{$C^{1,\alpha}$ domains}

When $L=\Delta$ in \eqref{pb} and $\Omega$ is $C^{k,\alpha}$, then solutions $u$ are as regular as the domain $\Omega$ provided that $f$ is regular enough.
In particular, if $\Omega$ is $C^\infty$ and $f\in C^\infty$ then $u\in C^\infty(\overline\Omega)$.

When $L=-(-\Delta)^s$, then the boundary regularity is well understood in $C^{1,1}$ and in $C^\infty$ domains.
In both cases, the optimal H\"older regularity of solutions is $u\in C^s(\overline\Omega)$, and in general one has $u\notin C^{s+\epsilon}(\overline\Omega)$ for any $\epsilon>0$.
Still, higher order estimates are given in terms of the regularity of $u/d^s$: if $\Omega$ is $C^\infty$ and $f\in C^\infty$ then $u/d^s\in C^\infty(\overline \Omega)$; see Grubb \cite{Grubb,Grubb2}.
Here, $d(x)={\rm dist}(x,\R^n\setminus\Omega)$.

We prove here a boundary regularity estimate of order $s+\alpha$ in $C^{1,\alpha}$ domains.
Namely, we show that if $\Omega$ is $C^{1,\alpha}$ then $u/d^s\in C^\alpha(\overline\Omega)$, as stated below.

We first establish the optimal H\"older regularity up to the boundary, $u\in C^s(\overline\Omega)$.

\begin{prop}\label{prop-Cs}
Let $s\in(0,1)$, $L$ be any operator of the form \eqref{operator-L}-\eqref{ellipt-const}, and $\Omega$ be any bounded $C^{1,\alpha}$ domain.
Let $u$ be a solution of \eqref{pb}.
Then,
\[\|u\|_{C^s(B_{1/2})}\leq C\left(\|f\|_{L^\infty(B_1\cap\Omega)}+\|u\|_{L^\infty(\R^n)}\right).\]
The constant $C$ depends only on $n$, $s$, $\Omega$, and ellipticity constants.
\end{prop}

Our second result gives a finer description of solutions in terms of the function~$d^s$, as explained above.

\begin{thm}\label{thm-bdry-reg}
Let $s\in(0,1)$ and $\alpha\in(0,s)$.
Let $L$ be any operator of the form \eqref{operator-L}-\eqref{ellipt-const}, $\Omega$ be any $C^{1,\alpha}$ domain, and $d$ be the distance to $\partial\Omega$.
Let $u$ be a solution of \eqref{pb}.
Then,
\[\|u/d^s\|_{C^{\alpha}(B_{1/2}\cap\overline\Omega)}\leq C\left(\|f\|_{L^\infty(B_1\cap\Omega)}+\|u\|_{L^\infty(\R^n)}\right).\]
The constant $C$ depends only on $n$, $s$, $\alpha$, $\Omega$, and ellipticity constants.
\end{thm}

The previous estimate in $C^{1,\alpha}$ domains was only known for the half-Laplacian $(-\Delta)^{1/2}$; see De Silva and Savin \cite{DS}.
For more general nonlocal operators, such estimate was only known in $C^{1,1}$ domains \cite{RS-stable}.

The proofs of Proposition \ref{prop-Cs} and Theorem \ref{thm-bdry-reg} follow the ideas of \cite{RS-stable}, where the same estimates were established in $C^{1,1}$ domains.
One of the main difficulties in the present proofs is the construction of appropriate barriers.
Indeed, while any $C^{1,1}$ domain satisfies the interior and exterior ball condition, this is not true anymore in $C^{1,\alpha}$ domains, and the construction of barriers is more delicate.
We will need a careful computation to show that
\[|L(d^s)|\leq Cd^{\alpha-s}\quad\textrm{in}\ \Omega.\]
In fact, since $d^s$ is not regular enough to compute $L$, we need to define a new function $\psi$ which behaves like $d$ but it is $C^2$ inside $\Omega$, and will show that $|L(\psi^s)|\leq Cd^{\alpha-s}$; see Definition \ref{d-is-psi}.

Once we have this, and doing some extra computations we will be able to construct sub and supersolutions which are comparable to $d^s$, and thus we will have
\[|u|\leq Cd^s.\]
This, combined with interior regularity estimates, will give the $C^s$ estimate of Proposition \ref{prop-Cs}.

Then, combining these ingredients with a blow-up and compactness argument in the spirit of \cite{RS-stable,RS-K}, we will find the expansion
\[\bigl|u(x)-Q(z)d^s(x)\bigr|\leq C|x-z|^{s+\alpha}\]
at any $z\in \partial\Omega$.
And this will yield Theorem \ref{thm-bdry-reg}.

\subsection{$C^1$ domains}

In $C^1$ domains, in general one does not expect solutions to be comparable to $d^s$.
In that case, we establish the following.

\begin{thm}\label{thmC1}
Let $s\in(0,1)$ and $\alpha\in(0,s)$.
Let $L$ be any operator of the form \eqref{operator-L}-\eqref{ellipt-const}, and $\Omega$ be any $C^1$ domain.

Then, there exists is $\delta>0$, depending only on $\alpha$, $n$, $s$, $\Omega$, and ellipticity constants, such that the following statement holds.

Let $u_1$ and $u_2$, be viscosity solutions of \eqref{pb} with right hand sides $f_1$ and $f_2$, respectively.
Assume that $\|f_i\|_{L^\infty(B_1\cap \Omega)}\leq C_0$ (with $C_0\geq\delta$), $\|u_i\|_{L^\infty(\R^n)}\leq C_0$,
\[f_i\geq-\delta\quad\textrm{in}\quad B_1\cap \Omega,\]
and that
\[u_i\geq0\quad\mbox{in}\quad \R^n,\qquad \sup_{B_{1/2}}u_i\geq 1.\]
Then,
\[ \|u_1/u_2\|_{C^\alpha(\Omega\cap B_{1/2})} \le CC_0,\qquad \alpha\in(0,s),\]
where $C$ depends only on $\alpha$, $n$, $s$, $\Omega$, and ellipticity constants.
\end{thm}

We expect the range of exponents $\alpha\in(0,s)$ to be optimal.

In particular, the previous result yields a boundary Harnack principle in $C^1$ domains.

\begin{cor}\label{corC1}
Let $s\in(0,1)$, $L$ be any operator of the form \eqref{operator-L}-\eqref{ellipt-const}, and $\Omega$ be any $C^1$ domain.
Let $u_1$ and $u_2$, be viscosity solutions of
\[\left\{ \begin{array}{rcll}
L u_1=Lu_2 &=&0&\textrm{in }B_1\cap \Omega \\
u_1=u_2&=&0&\textrm{in }B_1\setminus\Omega,
\end{array}\right.\]
Assume that
\[u_1\geq0\qquad\textrm{and}\qquad u_2\geq0\qquad\mbox{in}\quad \R^n,\]
and that $\sup_{B_{1/2}}u_1=\sup_{B_{1/2}}u_2=1$.
Then,
\[ 0<C^{-1}\leq \frac{u_1}{u_2}\leq C\qquad\textrm{in}\quad B_{1/2},\]
where $C$ depends only on $n$, $s$, $\Omega$, and ellipticity constants.
\end{cor}

Theorems \ref{thmC1} and \ref{thm-bdry-reg} will be important tools in a forthcoming work on free boundary problems for nonlocal elliptic operators \cite{CRS-obstacle}.
Namely, Theorem \ref{thmC1} (applied to the derivatives of the solution to the free boundary problem) will yield that $C^1$ free boundaries are in fact $C^{1,\alpha}$, and then thanks to Theorem \ref{thm-bdry-reg} we will get a fine description of solutions in terms of $d^s$.

\subsection{Equations with bounded measurable coefficients}

We also obtain estimates for equations with bounded measurable coefficients,
\begin{equation}\label{pb-2}
\left\{ \begin{array}{rcll}
M^+u &\geq&-K_0&\textrm{in }B_1\cap \Omega \\
M^-u &\leq&K_0&\textrm{in }B_1\cap \Omega \\
u&=&0&\textrm{in }B_1\setminus\Omega.
\end{array}\right.\end{equation}
Here, $M^+$ and $M^-$ are the extremal operators associated to the class $\mathcal L_*$, consisting of all operators of the form \eqref{operator-L}-\eqref{ellipt-const}, i.e.,
\[M^+:=M^+_{\mathcal L_*}u=\sup_{L\in\mathcal L_*}Lu,\qquad M^-:=M^+_{\mathcal L_*}u=\inf_{L\in\mathcal L_*}Lu.\]
Notice that the equation \eqref{pb-2} is an equation with bounded measurable coefficients, and it is the nonlocal analogue of
\[a_{ij}(x)\partial_{ij}u=f(x),\qquad \textrm{with}\quad \lambda{\rm Id}\leq (a_{ij}(x))_{ij}\leq \Lambda{\rm Id},\qquad |f(x)|\leq K_0.\]
For nonlocal equations with bounded measurable coefficients in $C^{1,\alpha}$ domains, we show the following.

Here, and throughout the paper, we denote $\bar\alpha=\bar\alpha(n,s,\lambda,\Lambda)>0$ the exponent in \cite[Proposition 5.1]{RS-K}.

\begin{thm}\label{thm-bdry-reg-2}
Let $s\in(0,1)$ and $\alpha\in(0,\bar\alpha)$.
Let $\Omega$ be any $C^{1,\alpha}$ domain, and $d$ be the distance to $\partial\Omega$.
Let $u\in C(B_1)$ be any viscosity solution of \eqref{pb-2}.
Then, we have
\[\|u/d^s\|_{C^{\alpha}(B_{1/2}\cap\overline\Omega)}\leq C\left(K_0+\|u\|_{L^\infty(\R^n)}\right).\]
The constant $C$ depends only on $n$, $s$, $\alpha$, $\Omega$, and ellipticity constants.
\end{thm}

In $C^1$ domains we prove:

\begin{thm}\label{thmC1-2}
Let $s\in(0,1)$ and $\alpha\in(0,\bar\alpha)$.
Let $\Omega$ be any $C^1$ domain.

Then, there exists is $\delta>0$, depending only on $\alpha$, $n$, $s$, $\Omega$, and ellipticity constants, such that the following statement holds.

Let $u_1$ and $u_2$, be functions satisfying
\[\left\{ \begin{array}{rcll}
M^+(au_1+bu_2) &\geq&-\delta(|a|+|b|)&\textrm{in }B_1\cap \Omega \\
u_1=u_2&=&0&\textrm{in }B_1\setminus\Omega,
\end{array}\right.\]
for any $a,b\in\R$.
Assume that
\[u_i\geq0\quad\mbox{in}\quad \R^n,\]
$\|u_i\|_{L^\infty(\R^n)}\leq C_0$, and that $\sup_{B_{1/2}}u_i\geq 1$.
Then, we have
\[ \|u_1/u_2\|_{C^\alpha(\Omega\cap B_{1/2})} \le C,\]
where $C$ depends only on $\alpha$, $n$, $s$, $\Omega$, and ellipticity constants.
\end{thm}

The Boundary Harnack principle for nonlocal operators has been widely studied, and in some cases it is even known in general open sets; see Bogdan \cite{Bogdan1}, Song-Wu \cite{SW}, Bogdan-Kulczycki-Kwasnicki \cite{Bogdan2}, and Bogdan-Kumagai-Kwasnicki \cite{Bogdan3}.
The main differences between our Theorems \ref{thmC1}-\ref{thmC1-2} and previous known results are the following.
On the one hand, our results allow a right hand side on the equation \eqref{pb}, and apply also to viscosity solutions of equations with bounded measurable coefficients \eqref{pb-2}.
On the other hand, we obtain a higher order estimate, in the sense that for linear equations we prove that $u_1/u_2$ is $C^\alpha$ for all $\alpha\in(0,s)$.
Finally, the proof we present here is perturbative, in the sense the we make a blow-up and use that after the rescaling the domain will be a half-space.
This allows us to get a higher order estimate for $u_1/u_2$, but requires the domain to be at least $C^1$.

\vspace{4mm}

The paper is organized as follows.
In Section \ref{sec-barriers} we construct the barriers in $C^{1,\alpha}$ domains.
Then, in Section \ref{sec-proof-main} we prove the regularity of solutions in $C^{1,\alpha}$ domains, that is, Proposition \ref{prop-Cs} and Theorems \ref{thm-bdry-reg} and \ref{thm-bdry-reg-2}.
In Section \ref{sec-barriers-C1} we construct the barriers needed in the analysis on $C^1$ domains.
Finally, in Section \ref{sec-regularity-C1} we prove Theorems \ref{thmC1} and \ref{thmC1-2}.

\section{Barriers: $C^{1,\alpha}$ domains}
\label{sec-barriers}

Throughout this section, $\Omega$ will be any bounded and $C^{1,\alpha}$ domain, and
\[d(x)={\rm dist}(x,\R^n\setminus\Omega).\]
Since $d$ is only $C^{1,\alpha}$ inside $\Omega$, we need to consider the following ``regularized version'' of $d$.

\begin{defi}\label{d-is-psi}
Given a $C^{1,\alpha}$ domain $\Omega$, we consider a fixed function $\psi$ satisfying
\begin{equation}\label{psi-d}
C^{-1}d\leq \psi\leq Cd,
\end{equation}
\begin{equation}\label{psi-d-2}
\|\psi\|_{C^{1,\alpha}(\overline\Omega)}\leq C\qquad\textrm{and}\qquad |D^2\psi|\leq Cd^{\alpha-1},
\end{equation}
with $C$ depending only on $\Omega$.
\end{defi}

\begin{rem}\label{rem-psi}
Notice that to construct $\psi$ one may take for example the solution to $-\Delta \psi=1$ in $\Omega$, $\psi=0$ on $\partial\Omega$, extended by $\psi=0$ in $\R^n\setminus\Omega$.

Note also that any $C^{1,\alpha}$ domain $\Omega$ can be locally represented as the epigraph of a $C^{1,\alpha}$ function.
More precisely, there is a $\rho_0>0$ such that for all $z\in\partial\Omega$ the set $\partial\Omega\cap B_{\rho_0}(z)$ is, after a rotation, the graph of a $C^{1,\alpha}$ function.
Then, the constant $C$ in \eqref{psi-d}-\eqref{psi-d-2} can be taken depending only on $\rho_0$ and on the $C^{1,\alpha}$ norms of these functions.
\end{rem}

We want to show the following.

\begin{prop}\label{lem11}
Let $s\in(0,1)$ and $\alpha\in(0,s)$, $L$ be given by \eqref{operator-L}-\eqref{ellipt-const}, and $\Omega$ be any $C^{1,\alpha}$ domain.
Let $\psi$ be given by Definition \ref{d-is-psi}.
Then,
\begin{equation}\label{eq-psi-1}
|L(\psi^s)| \leq Cd^{\alpha-s}\qquad \textrm{in}\ \Omega.
\end{equation}
The constant $C$ depends only on $s$, $n$, $\Omega$, and ellipticity constants.
\end{prop}

For this, we need a couple of technical Lemmas.
The first one reads as follows.

\begin{lem}\label{lem13}
Let $\Omega$ be any $C^{1,\alpha}$ domain, and $\psi$ be given by Definition \ref{d-is-psi}.
Then, for each $x_0\in \Omega$ we have
\[\left|\psi(x_0+y)-\bigl(\psi(x_0)+\nabla \psi(x_0)\cdot y\bigr)_+\right|\leq C|y|^{1+\alpha}\qquad \textrm{for}\ y\in \R^n.\]
The constant $C$ depends only on $\Omega$.
\end{lem}

\begin{proof}
Let us consider $\tilde \psi$, a $C^{1,\alpha}(\R^n)$ extension of $\psi|_\Omega$ satisfying $\tilde\psi\leq 0$ in $\R^n\setminus\Omega$.
Then, since $\tilde \psi\in C^{1,\alpha}(\R^n)$ we clearly have
\[\left|\tilde\psi(x)-\psi(x_0)-\nabla\psi(x_0)\cdot (x-x_0)\right|\leq C|x-x_0|^{1+\alpha}\]
in all of $\R^n$. Here we used $\tilde \psi(x_0)=\psi(x_0)$ and $\nabla \tilde \psi(x_0)=\nabla\psi(x_0)$.

Now, using that $|a_+-b_+|\leq |a-b|$, combined with $(\tilde\psi)_+=\psi$, we find
\[\left|\psi(x)-\bigl(\psi(x_0)+\nabla \psi(x_0)\cdot (x-x_0)\bigr)_+\right|\leq C|x-x_0|^{1+\alpha}\]
for all $x\in \R^n$.
Thus, the lemma follows.
\end{proof}

The second one reads as follows.

\begin{lem}\label{lem15}
Let $\Omega$ be any $C^{1,\alpha}$ domain, $p\in \Omega$, and $\rho=d(p)/2$.
Let $\gamma>-1$ and $\beta\neq\gamma$.
Then,
\[\int_{B_1\setminus B_{\rho/2}}d^\gamma(p+y)\frac{dy}{|y|^{n+\beta}}\leq C\bigl(1+\rho^{\gamma-\beta}\bigr).\]
The constant $C$ depends only on $\gamma$, $\beta$, and $\Omega$.
\end{lem}

\begin{proof}
The proof is similar to that of \cite[Lemma 4.2]{RV}.

First, we may assume $p=0$.

Notice that, since $\Omega$ is $C^{1,\alpha}$, then there is $\kappa_*>0$ such that for any $t\in(0,\kappa_*]$ the level set $\{d=t\}$ is $C^{1,\alpha}$.
Since
\begin{equation}\label{kappa1}
\int_{(B_1\setminus B_\rho)\cap \{d\geq\kappa_*\}}d^\gamma(y)\frac{dy}{|y|^{n+\beta}}\leq C,
\end{equation}
then we just have to bound the same integral in the set $\{d<\kappa_*\}$.
Here we used that $B_r\cap \{d\geq\kappa_*\}=\emptyset$ if $r\leq \kappa_*-2\rho$, which follows from the fact that $d(0)=2\rho$.

We will use the following estimate for $t\in (0,\kappa_*)$
\[\mathcal{H}^{n-1}\bigl(\{d=t\}\cap (B_{2^{-k+1}}\setminus B_{2^{-k}})\bigr)\leq C(2^{-k})^{n-1},\]
which follows for example from the fact that $\{d=t\}$ is $C^{1,\alpha}$ (see the Appendix in \cite{RV}).
Note also that $\{d=t\}\cap B_r=\emptyset$ if $t>r+2\rho$.

Let $M\geq0$ be such that $2^{-M}\leq \rho\leq 2^{-M+1}$.
Then, using the coarea formula,
\begin{equation}\label{kappa2}\begin{split}
\int_{(B_1\setminus B_\rho)\cap \{d<\kappa_*\}}& d^\gamma(y)\frac{dy}{|y|^{n+\beta}}\leq \\
&\leq\ \sum_{k=0}^M \frac{1}{2^{-k(n+\beta)}}\int_{(B_{2^{-k+1}}\setminus B_{2^{-k}})\cap \{d<C2^{-k}\}}
    d^\gamma(y)|\nabla d(y)|\,dy \\
&\leq\  \sum_{k=0}^M \frac{1}{2^{-k(n+\beta)}}\int_0^{C2^{-k}}t^\gamma dt\int_{(B_{2^{-k+1}}\setminus B_{2^{-k}})\cap \{d=t\}}
    d\mathcal{H}^{n-1}(y)  \\
&\leq\  C\sum_{k=0}^M \frac{(2^{-k})^{\gamma+1}2^{-k(n-1)}}{2^{-k(n+\beta)}}= C\sum_{k=0}^M 2^{k(\beta-\gamma)}= C(1+\rho^{\gamma-\beta}).
\end{split}\end{equation}
Here we used that $\gamma\neq \beta$ ---in case $\gamma=\beta$ we would get $C(1+|\log \rho|)$.

Combining \eqref{kappa1} and \eqref{kappa2}, the lemma follows.
\end{proof}

We now give the:

\begin{proof}[Proof of Proposition \ref{lem11}]
Let $x_0\in \Omega$ and $\rho=d(x)$.

Notice that when $\rho\geq \rho_0>0$ then $\psi^s$ is smooth in a neighborhood of $x_0$, and thus $L(\psi^s)(x_0)$ is bounded by a constant depending only on $\rho_0$.
Thus, we may assume that $\rho\in(0,\rho_0)$, for some small $\rho_0$ depending only on $\Omega$.

Let us denote
\[\ell(x)=\bigl(\psi(x_0)+\nabla \psi(x_0)\cdot (x-x_0)\bigr)_+,\]
which satisfies
\[L(\ell^s)=0\qquad\textrm{in}\quad \{\ell>0\};\]
see \cite[Section 2]{RS-K}.

Now, notice that
\[\psi(x_0)=\ell(x_0)\qquad{\rm and}\qquad \nabla\psi(x_0)=\nabla \ell(x_0).\]
Moreover, by Lemma \ref{lem13} we have
\[\left|\psi(x_0+y)-\ell(x_0+y)\right|\leq C|y|^{1+\alpha},\]
and using $|a^s-b^s|\leq C|a-b|(a^{s-1}+b^{s-1})$ for $a,b\geq0$, we find
\begin{equation}\label{bound-1}
\left|\psi^s(x_0+y)-\ell^s(x_0+y)\right|\leq C|y|^{1+\alpha}\left(d^{s-1}(x_0+y)+\ell^{s-1}(x_0+y)\right).
\end{equation}
Here, we used that $\psi\leq Cd$.

On the other hand, since $\psi\in C^{1,\alpha}(\overline\Omega)$ and $\psi\geq cd$ in $\overline\Omega$, then it is not difficult to check that
\[\ell>0\quad \textrm{in}\quad B_{\rho/2}(x_0),\]
provided that $\rho_0$ is small (depending only on $\Omega$).
Thanks to this, one may estimate
\[\left|D^2(\psi^s-\ell^s)\right|\leq C\rho^{s+\alpha-2}\quad \textrm{in}\ B_{\rho/2},\]
and thus
\begin{equation}\label{bound-2}
\bigl|\psi^s-\ell^s\bigr|(x_0+y)\leq \|D^2(\psi^s-\ell^s)\|_{L^\infty(B_{\rho/2}(x_0))}|y|^2 \leq C\rho^{s+\alpha-2}|y|^2
\end{equation}
for $y\in B_{\rho/2}$.

Therefore, it follows from \eqref{bound-1} and \eqref{bound-2} that
\[\bigl|\psi^s-\ell^s\bigr|(x_0+y) \le
\begin{cases}
C \rho^{s+\alpha-2} |y|^2 \hfill\mbox{for } y\in B_{\rho/2}\\
C |y|^{1+\alpha}\left(d^{s-1}(x_0+y)+\ell^{s-1}(x_0+y)\right)  \quad \mbox{for } y\in B_1\setminus B_{\rho/2} \\
C  |y|^s              \hfill  \mbox{for } y\in \R^n \setminus B_1.
\end{cases}\]

Hence, recalling that $L(\ell^s)(x_0)=0$, we find
\[
\begin{split}
|L(\psi^s)(x_0)| &=  |L\bigl(\psi^s-\ell^s)(x_0)|
\\
&=\int_{\R^n} \bigl|\psi^s-\ell^s\bigr|(x_0+y) \frac{a(y/|y|)}{|y|^{n+2s}}\,dy
\\
&\le  \int_{B_{\rho/2}} C\rho^{s+\alpha-2}|y|^2\frac{dy}{|y|^{n+2s}}+\int_{\R^n\setminus B_1}C|y|^s \frac{dy}{|y|^{n+2s}}\,+\\
&\quad +\int_{B_1\setminus B_{\rho/2}} C|y|^{1+\alpha}\left(d^{s-1}(x_0+y)+\ell^{s-1}(x_0+y)\right)\frac{dy}{|y|^{n+2s}}
\\
&\le   C(\rho^{\alpha-s}+1)+C\int_{B_1\setminus B_{\rho/2}}\left(d^{s-1}(x_0+y)+\ell^{s-1}(x_0+y)\right)\frac{dy}{|y|^{n+2s-1-\alpha}}.
\end{split}
\]
Thus, using Lemma \ref{lem15} twice, we find
\[|L(\psi^s)(x_0)|\leq C\rho^{\alpha-s},\]
and \eqref{eq-psi-1} follows.
\end{proof}

When $\alpha>s$ the previous proof gives the following result, which states that for any operator \eqref{operator-L}-\eqref{ellipt-const} one has $L(d^s)\in L^\infty(\Omega)$.
Here, as in \cite{Grubb,RS-stable,RS-K}, $d$ denotes a fixed function that coincides with ${\rm dist}(x,\R^n\setminus\Omega)$ in a neighborhood of $\partial\Omega$, satisfies $d\equiv0$ in $\R^n\setminus\Omega$, and it is $C^{1,\alpha}$ in $\Omega$.

\begin{prop}\label{d^s-is-sol}
Let $s\in(0,1)$, $L$ be given by \eqref{operator-L}-\eqref{ellipt-const}, and $\Omega$ be any bounded $C^{1,\alpha}$ domain, with $\alpha>s$.
Then,
\[|L(d^s)|\leq C\quad \textrm{in}\ \Omega.\]
The constant $C$ depends only on $n$, $s$, $\Omega$, and ellipticity constants.
\end{prop}

To our best knowledge, this result was only known in case that $L$ is the fractional Laplacian and $\Omega$ is $C^{1,1}$, or in case that $a\in C^\infty(S^{n-1})$ in \eqref{operator-L} and $\Omega$ is $C^\infty$ (in this case $L(d^s)$ is $C^\infty(\overline\Omega)$; see \cite{Grubb}).

Also, recall that for a general stable operator \eqref{operator-L} (with $a\in L^1(S^{n-1})$ and without the assumption \eqref{ellipt-const}) the result is false, since we constructed in \cite{RS-stable} an operator $L$ and a $C^\infty$ domain~$\Omega$ for which $L(d^s)\notin L^\infty(\Omega)$.
Hence, the assumption \eqref{ellipt-const} is somewhat necessary for Proposition \ref{d^s-is-sol} to be true.

\begin{proof}[Proof of Proposition \ref{d^s-is-sol}]
Let $x_0\in \Omega$, and $\rho=d(x)$.

Notice that when $\rho\geq \rho_0>0$ then $d^s$ is $C^{1+s}$ in a neighborhood of $x_0$, and thus $L(d^s)(x_0)$ is bounded by a constant depending only on $\rho_0$.
Thus, we may assume that $\rho\in(0,\rho_0)$, for some small $\rho_0$ depending only on $\Omega$.

Let us denote
\[\ell(x)=\bigl(d(x_0)+\nabla d(x_0)\cdot (x-x_0)\bigr)_+,\]
which satisfies
\[L(\ell^s)=0\qquad\textrm{in}\quad \{\ell>0\}.\]
Moreover, as in Proposition \ref{lem11}, we have
\begin{equation}\label{bound-11}
|d^s(x_0+y)-\ell^s(x_0+y)|\leq C|y|^{1+\alpha}\left(d^{s-1}(x_0+y)+\ell^{s-1}(x_0+y)\right).
\end{equation}
In particular,
\[|d^s(x_0+y)-\ell^s(x_0+y)|\leq C\rho^{s-1}|y|^{1+\alpha}\qquad\textrm{for}\ y\in B_{\rho/2}.\]

Hence, recalling that $L(\ell^s)(x_0)=0$, we find
\[
\begin{split}
|L(\psi^s)(x_0)| &=  |L\bigl(\psi^s-\ell^s)(x_0)|
\\
&=\int_{\R^n} \bigl|\psi^s-\ell^s\bigr|(x_0+y) \frac{a(y/|y|)}{|y|^{n+2s}}\,dy
\\
&\le  \int_{B_{\rho/2}} C\rho^{s-1}|y|^{1+\alpha}\frac{dy}{|y|^{n+2s}}+\int_{\R^n\setminus B_1}C|y|^s \frac{dy}{|y|^{n+2s}}\,+\\
&\quad +\int_{B_1\setminus B_{\rho/2}} C|y|^{1+\alpha}\left(d^{s-1}(x_0+y)+\ell^{s-1}(x_0+y)\right)\frac{dy}{|y|^{n+2s}}
\\
&\le   C(1+\rho^{\alpha-s}).
\end{split}
\]
Here we used Lemma \ref{lem15}.
Since $\alpha>s$, the result follows.
\end{proof}

We next show the following.

\begin{lem}\label{lem12}
Let $s\in(0,1)$, $L$ be given by \eqref{operator-L}-\eqref{ellipt-const}, and $\Omega$ be any $C^{1,\alpha}$ domain.
Let $\psi$ be given by Definition~\ref{d-is-psi}.
Then, for any $\epsilon\in(0,\alpha)$, we have
\begin{equation}\label{eq-psi-2}
L(\psi^{s+\epsilon})\geq cd^{\epsilon-s}-C\qquad \textrm{in}\ \Omega\cap B_{1/2},
\end{equation}
with $c>0$.
The constants $c$ and $C$ depend only on $\epsilon$, $s$, $n$, $\Omega$, and ellipticity constants.
\end{lem}

\begin{proof}
Exactly as in Proposition \ref{lem11}, one finds that
\begin{equation}\label{bound-1-eps}
\left|\psi^{s+\epsilon}(x_0+y)-\ell^{s+\epsilon}(x_0+y)\right|\leq C|y|^{1+\alpha}\left(d^{s+\epsilon-1}(x_0+y)+\ell^{s+\epsilon-1}(x_0+y)\right),
\end{equation}
and
\begin{equation}\label{bound-2-eps}
\bigl|\psi^{s+\epsilon}-\ell^{s+\epsilon}\bigr|(x_0+y)\leq C\rho^{s+\epsilon+\alpha-2}|y|^2
\end{equation}
for $y\in B_{\rho/2}$.
Therefore, as in Proposition \ref{lem11},
\[|L\bigl(\psi^{s+\epsilon}-\ell^{s+\epsilon})(x_0)| \le  C(1+\rho^{\alpha+\epsilon-s}).\]
We now use that, by homogeneity, we have
\[L(\ell^{s+\epsilon})(x_0)=\kappa\rho^{\epsilon-s},\]
with $\kappa>0$ (see \cite{RS-K}).
Thus, combining the previous two inequalities we find
\[L(\psi^{s+\epsilon})(x_0)\geq \kappa\rho^{\epsilon-s}-C(1+\rho^{\alpha+\epsilon-s})\geq \frac{\kappa}{2}\rho^{s-\epsilon}-C,\]
as desired.
\end{proof}

We now construct sub and supersolutions.

\begin{lem}[Supersolution]\label{supersol}
Let $s\in(0,1)$, $L$ be given by \eqref{operator-L}-\eqref{ellipt-const}, and $\Omega$ be any bounded $C^{1,\alpha}$ domain.
Then, there exists $\rho_0>0$ and a function $\phi_1$ satisfying
\[\left\{ \begin{array}{rcll}
L \phi_1 &\leq&-1&\textrm{in }\Omega\cap \{d\leq \rho_0\} \\
C^{-1}d^s\ \leq \ \phi_1 &\leq &Cd^s&\textrm{in }\Omega \\
\phi_1&=&0&\textrm{in }\R^n\setminus\Omega.
\end{array}\right.\]
The constants $C$ and $\rho_0$ depend only on $n$, $s$, $\Omega$, and ellipticity constants.
\end{lem}

\begin{proof}
Let $\psi$ be given by Definition \ref{d-is-psi}, and let $\epsilon=\frac\alpha2$.
Then, by Proposition \ref{lem11} we have
\[-C_0d^{\alpha-s}\leq L(\psi^s)\leq C_0d^{\alpha-s},\]
and by Lemma \ref{lem12}
\[L(\psi^{s+\epsilon})\geq c_0d^{\epsilon-s}-C_0.\]

Next, we consider the function
\[\phi_1=\psi^s-c\psi^{s+\epsilon},\]
with $c$ small enough.
Then, $\phi_1$ satisfies
\begin{equation}\label{mecut}
L\phi_1\leq C_0d^{\alpha-s}+C_0-cc_1d^{\epsilon-s}\leq -1\qquad\textrm{in}\quad \Omega\cap\{d\leq \rho_0\},
\end{equation}
for some $\rho_0>0$.
Finally, by construction we clearly have
\[C^{-1}d^s\leq \phi_1\leq Cd^s\quad\textrm{in}\quad\Omega,\]
and thus the Lemma is proved.
\end{proof}

Notice that the previous proof gives in fact the following.

\begin{lem}\label{supersol-sing-RHS}
Let $s\in(0,1)$, $L$ be given by \eqref{operator-L}-\eqref{ellipt-const}, and $\Omega$ be any bounded $C^{1,\alpha}$ domain.
Then, there exist $\rho_0>0$ and a function $\phi_1$ satisfying
\[\left\{ \begin{array}{rcll}
L \phi_1 &\leq&-d^{\epsilon-s}&\textrm{in }\Omega\cap \{d\leq\rho_0\} \\
C^{-1}d^s\ \leq \ \phi_1 &\leq &Cd^s&\textrm{in }\Omega \\
\phi_1&=&0&\textrm{in }\R^n\setminus\Omega.
\end{array}\right.\]
The constants $C$ and $\rho_0$ depend only on $n$, $s$, $\Omega$, and ellipticity constants.
\end{lem}

\begin{proof}
The proof is the same as Lemma \ref{supersol}; see \eqref{mecut}.
\end{proof}

We finally construct a subsolution.

\begin{lem}[Subsolution]
Let $s\in(0,1)$, $L$ be given by \eqref{operator-L}-\eqref{ellipt-const}, and $\Omega$ be any bounded $C^{1,\alpha}$ domain.
Then, for each $K\subset\subset\Omega$ there exists a function $\phi_2$ satisfying
\[\left\{ \begin{array}{rcll}
L \phi_2 &\geq&1&\textrm{in }\Omega\setminus K \\
C^{-1}d^s\ \leq\ \phi_2 &\leq &Cd^s&\textrm{in }\Omega \\
\phi_2&=&0&\textrm{in }\R^n\setminus\Omega.
\end{array}\right.\]
The constants $c$ and $C$ depend only on $n$, $s$, $\Omega$, $K$, and ellipticity constants.
\end{lem}

\begin{proof}
First, notice that if $\eta\in C^\infty_c(K)$ then $L\eta\geq c_1>0$ in $\Omega\setminus K$.
Hence,
\[\phi_2=\psi^s+\psi^{s+\epsilon}+C\eta\]
satisfies
\[L\phi_2\geq -C_0d^{\alpha-s}+c_0d^{\epsilon-s}-C_0+Cc_1\geq 1\qquad \textrm{in}\quad \Omega\setminus K,\]
provided that $C$ is chosen large enough.
\end{proof}

\section{Regularity in $C^{1,\alpha}$ domains}
\label{sec-proof-main}

The aim of this section is to prove Proposition \ref{prop-Cs} and Theorem \ref{thm-bdry-reg}.

\subsection{H\"older regularity up to the boundary}

We will prove first the following result, which is similar to Proposition \ref{prop-Cs} but allows $u$ to grow at infinity and $f$ to be singular near $\partial\Omega$.

\begin{prop}\label{prop-Cs-2}
Let $s\in(0,1)$, $L$ be any operator of the form \eqref{operator-L}-\eqref{ellipt-const}, and $\Omega$ be any bounded $C^{1,\alpha}$ domain.
Let $u$ be a solution to \eqref{pb}, and assume that
\[|f|\leq Cd^{\epsilon-s}\quad\textrm{in}\quad \Omega.\]
Then,
\[\|u\|_{C^s(B_{1/2})}\leq C\left(\|d^{s-\epsilon}f\|_{L^\infty(B_1\cap\Omega)}+\sup_{R\geq1}R^{\delta-2s}\|u\|_{L^\infty(B_R)}\right).\]
The constant $C$ depends only on $n$, $s$, $\epsilon$, $\delta$, $\Omega$, and ellipticity constants.
\end{prop}

\begin{proof}
Dividing by a constant, we may assume that
\[\|d^{s-\epsilon}f\|_{L^\infty(B_1\cap\Omega)}+\sup_{R\geq1}R^{\delta-2s}\|u\|_{L^\infty(B_R)}\leq 1.\]
Then, the truncated function $w=u\chi_{B_1}$ satisfies
\[|Lw|\leq Cd^{\epsilon-s}\quad \textrm{in}\quad \Omega\cap B_{3/4},\]
$w\leq 1$ in $B_1$, and $w\equiv0$ in $\R^n\setminus B_1$.

Let $\widetilde\Omega$ be a bounded $C^{1,\alpha}$ domain satisfying: $B_1\cap\Omega\subset \widetilde\Omega$; $B_{1/2}\cap \partial\Omega\subset \partial\widetilde\Omega$; and ${\rm dist}(x,\partial\widetilde\Omega)\geq c>0$ in $\Omega\cap (B_1\setminus B_{3/4})$.
Let $\phi_1$ be the function given by Lemma \ref{supersol}, satisfying
\[\left\{ \begin{array}{rcll}
L \phi_1 &\leq&-\tilde d^{\epsilon-s}&\textrm{in }\widetilde\Omega\cap \{\tilde d\leq \rho_0\} \\
c\tilde d^s\ \leq \ \phi_1 &\leq &C\tilde d^s&\textrm{in }\widetilde\Omega \\
\phi_1&=&0&\textrm{in }\R^n\setminus\Omega,
\end{array}\right.\]
where we denoted $\tilde d(x)={\rm dist}(x,\R^n\setminus\widetilde\Omega)$.

Then, the function $\varphi=C\phi_1$ satisfies
\[\left\{ \begin{array}{rcll}
L \varphi &\leq&-Cd^{\epsilon-s}&\textrm{in }\Omega\cap B_{1/2}\cap  \{d\leq \rho_0\} \\
\varphi &\leq &Cd^s&\textrm{in }\Omega\cap B_{1/2} \\
\varphi &\geq &1&\textrm{in }\Omega\cap (B_1\setminus B_{3/4})\quad\textrm{and in}\quad \Omega\cap B_{1/2}\cap \{d\geq\rho_0\} \\
\varphi&\geq&0&\textrm{in }\R^n.
\end{array}\right.\]
In particular, if $C$ is large enough then we have $L(\varphi-w)\leq0$ in $\Omega\cap B_{1/2}\cap \{d\leq \rho_0\}$, and $\varphi-w\geq0$ in $\R^n\setminus (\Omega\cap B_{1/2}\cap \{d\leq\rho_0\})$.

Therefore, the maximum principle yields $w\leq \varphi$, and thus $w\leq Cd^s$ in $B_{1/2}$.
Replacing $w$ by $-w$, we find
\begin{equation}\label{d^s}
|w|\leq Cd^s\qquad \textrm{in}\quad B_{1/2}.
\end{equation}

Now, it follows from the interior estimates of \cite[Theorem 1.1]{RS-stable} that
\[r^{s}[w]_{C^s(B_r(x_0))}\leq C\bigl(r^{2s}\|Lw\|_{L^\infty(B_{2r}(x_0))}+\sup_{R\geq1}R^{\delta-2s}\|w\|_{L^\infty(B_{rR}(x_0))}\bigr)\]
for any ball $B_r(x_0)\subset \Omega\cap B_{1/2}$ with $2r=d(x_0)$.
Now, taking $\delta=s$ and using \eqref{d^s}, we find
\[R^{-s}\|w\|_{L^\infty(B_{rR}(x_0))}\leq Cr^s\quad\textrm{for all}\ R\geq1.\]
Thus, we have
\[[w]_{C^s(B_r(x_0))}\leq C\]
for all balls $B_r(x_0)\subset \Omega\cap B_{1/2}$ with $2r=d(x_0)$.
This yields
\[\|w\|_{C^s(B_{1/2})}\leq C.\]
Indeed, take $x,y\in B_{1/2}$, let $r=|x-y|$ and $\rho=\min\{d(x),d(y)\}$.
If $2\rho\geq r$, then using $|u|\leq Cd^s$
\[|u(x)-u(y)|\leq |u(x)|+|u(y)|\leq Cr^s+C(r+\rho)^s\leq \bar C\rho^s.\]
If $2\rho<r$ then $B_{2\rho}(x)\subset\Omega$, and hence
\[|u(x)-u(y)|\leq \rho^s[u]_{C^s(B_\rho(x))}\leq C\rho^s.\]
Thus, the proposition is proved.
\end{proof}

The proof of Proposition is now immediate.

\begin{proof}[Proof of Proposition \ref{prop-Cs}]
The result is a particular case of Proposition \ref{prop-Cs-2}.
\end{proof}

\subsection{Regularity for $u/d^s$}

Let us now prove Theorem \ref{thm-bdry-reg}.
For this, we first show the following.

\begin{prop}\label{prop-bdry-reg}
Let $s\in(0,1)$ and $\alpha\in(0,s)$.
Let $L$ be any operator of the form \eqref{operator-L}-\eqref{ellipt-const}, $\Omega$ be any $C^{1,\alpha}$ domain, and $\psi$ be given by Definition \ref{d-is-psi}.

Assume that $0\in \partial\Omega$, and that $\partial\Omega\cap B_1$ can be represented as the graph of a $C^{1,\alpha}$ function with norm less or equal than~1.

Let $u$ be any solution to \eqref{pb}, and let
\[K_0=\|d^{s-\alpha}f\|_{L^\infty(B_1\cap\Omega)}+\|u\|_{L^\infty(\R^n)}.\]
Then, there exists a constant $Q$ satisfying $|Q|\leq CK_0$ and
\[\bigl|u(x)-Q\psi^s(x)\bigr|\leq CK_0|x|^{s+\alpha}.\]
The constant $C$ depends only on $n$, $s$, and ellipticity constants.
\end{prop}

We will need the following technical lemma.

\begin{lem}\label{lem-blow-up-seq}
Let $\Omega$, $\psi$, and $u$ be as in Proposition \ref{prop-bdry-reg}, and define
\begin{equation}\label{def-phi}
\phi_r(x):=Q_*(r)\psi^s(x),
\end{equation}
where
\[Q_*(r):=\arg\min_{Q\in\R}\int_{B_r}\bigl(u-Q\psi^2\bigr)^2dx=\frac{\int_{B_r}u\psi^s}{\int_{B_r}\psi^{2s}dx}.\]

Assume that for all $r\in(0,1)$ we have
\begin{equation}\label{5.2}
\|u-\phi_r\|_{L^\infty(B_r)}\leq C_0r^{s+\alpha}.
\end{equation}
Then, there is $Q\in \R$ satisfying $|Q|\leq C(C_0+\|u\|_{L^\infty(B_1)})$ such that
\[\|u-Q\psi^s\|_{L^\infty(B_r)}\leq CC_0r^{s+\alpha},\]
for some constant $C$ depending only on $s$ and $\alpha$.
\end{lem}

\begin{proof}
The proof is analogue to that of \cite[Lemma 5.3]{RS-stable}.

First, we may assume $C_0+\|u\|_{L^\infty(B_1)}=1$.
Then, by \eqref{5.2}, for all $x\in B_r$ we have
\[|\phi_{2r}(x)-\phi_r(x)|\leq |u(x)-\phi_{2r}(x)|+|u(x)-\phi_r(x)|\leq Cr^{s+\alpha}.\]
This, combined with $\sup_{B_r}\psi^s=cr^s$, gives
\[|Q_*(2r)-Q_*(r)|\leq Cr^\alpha.\]
Moreover, we have $|Q_*(1)|\leq C$, and thus there exists the limit $Q=\lim_{r\downarrow0}Q_*(r)$.
Furthermore,
\[|Q-Q_*(r)|\leq \sum_{k\geq0}|Q_*(2^{-k}r)-Q_*(2^{-k-1}r)|\leq \sum_{k\geq0}C2^{-m\alpha}r^\alpha\leq Cr^\alpha.\]
In particular, $|Q|\leq C$.

Therefore, we finally find
\[\|u-Q\psi^s\|_{L^\infty(B_r)}\leq \|u-Q_*(r)\psi^s\|_{L^\infty(B_r)}+Cr^s|Q_*(r)-Q|\leq Cr^{s+\alpha},\]
and the lemma is proved.
\end{proof}

We now give the:

\begin{proof}[Proof of Proposition \ref{prop-bdry-reg}]
The proof is by contradiction, and uses several ideas from \cite[Section 5]{RS-stable}.

First, dividing by a constant we may assume $K_0=1$.
Also, after a rotation we may assume that the unit (outward) normal vector to $\partial\Omega$ at $0$ is $\nu=-e_n$.

Assume the estimate is not true, i.e., there are sequences $\Omega_k$, $L_k$, $f_k$, $u_k$, for which:
\begin{itemize}
\item $\Omega_k$ is a $C^{1,\alpha}$ domain that can be represented as the graph of a $C^{1,\alpha}$ function with norm is less or equal than 1;
\item $0\in \partial\Omega_k$ and the unit normal vector to $\partial\Omega_k$ at $0$ is $-e_n$;
\item $L_k$ is of the form \eqref{operator-L}-\eqref{ellipt-const};
\item $\|d^{s-\alpha}f_k\|_{L^\infty(B_1\cap\Omega)}+\|u_k\|_{L^\infty(\R^n)}\leq 1$;
\item For any constant $Q$, $\sup_{r>0}\sup_{B_r}r^{-s-\alpha}|u_k-Q\psi_k^s|=\infty$.
\end{itemize}
Then, by Lemma \ref{lem-blow-up-seq} we will have
\[\sup_k\sup_{r>0}\|u_k-\phi_{k,r}\|_{L^\infty(B_r)}=\infty,\]
where
\[\phi_{k,r}(x)=Q_k(r)\psi_k^s,\qquad Q_k(r)=\frac{\int_{B_r}u_k\psi_k^s}{\int_{B_r}\psi_k^{2s}}.\]

We now define the monotone quantity
\[\theta(r):=\sup_k\sup_{r'>r}(r')^{-s-\alpha}\|u_k-\phi_{k,r'}\|_{L^\infty(B_{r'})},\]
which satisfies $\theta(r)\to\infty$ as $r\to0$.
Hence, there are sequences $r_m\to0$ and $k_m$, such that
\begin{equation}\label{theta-nondeg}
(r_m)^{-s-\alpha}\|u_{k_m}-\phi_{k_m,r_m}\|_{L^\infty(B_{r_m})}\geq \frac12\theta(r_m).
\end{equation}

Let us now denote $\phi_m=\phi_{k_m,r_m}$ and define
\[v_m(x):=\frac{u_{k_m}(r_mx)-\phi_m(r_mx)}{(r_m)^{s+\alpha}\theta(r_m)}.\]
Note that
\begin{equation}\label{vm-orth}
\int_{B_1} v_m(x)\psi_k^s(r_mx)dx=0,
\end{equation}
and also
\begin{equation}\label{vm-nondeg}
\|v_m\|_{L^\infty(B_1)}\geq\frac12,
\end{equation}
which follows from \eqref{theta-nondeg}.

With the same argument as in the proof of Lemma \ref{lem-blow-up-seq}, one finds
\[|Q_{k_m}(2r)-Q_{k_m}(r)|\leq Cr^\alpha\theta(r).\]
Then, by summing a geometric series this yields
\[|Q_{k_m}(rR)-Q_{k_m}(r)|\leq Cr^\alpha\theta(r) R^\alpha\]
for all $R\geq1$ (see \cite{RS-stable}).

The previous inequality, combined with
\[\|u_m-Q_{k_m}(r_mR)\psi_{k_m}^s\|_{L^\infty(B_{r_mR})}\leq (r_mR)^{s+\alpha}\theta(r_mR)\]
(which follows from the definition of $\theta$), gives
\begin{equation}\label{vm-growth}\begin{split}
\|v_m\|_{L^\infty(B_R)}&=\frac{1}{(r_m)^{s+\alpha}\theta(r_m)}\|u_m-Q_{k_m}(r_m)\psi_{k_m}^s\|_{L^\infty(B_{r_mR})}\\
&\leq \frac{(r_mR)^{s+\alpha}\theta(r_mR)}{(r_m)^{s+\alpha}\theta(r_m)}+
\frac{C(r_mR)^s}{(r_m)^{s+\alpha}\theta(r_m)}|Q_{k_m}(r_mR)-Q_{k_m}(r_m)|\\
&\leq CR^{s+\alpha}
\end{split}\end{equation}
for all $R\geq1$.
Here we used that $\theta(r_mR)\leq \theta(r_m)$ if $R\geq1$.

Now, the functions $v_m$ satisfy
\[L_m v_m(x)=\frac{(r_m)^{2s}}{(r_m)^{s+\alpha}\theta(r_m)}\,f_{k_m}(r_m x)-\frac{(r_m)^{2s}}{(r_m)^{s+\alpha}\theta(r_m)}\,(L\psi_{k_m})(r_m x)\]
in $(r_m^{-1}\Omega_{k_m})\cap B_{r_m^{-1}}$.
Since $\alpha<s$, and using Proposition \ref{lem11}, we find
\[|L_m v_m|\leq \frac{C}{\theta(r_m)}\,(r_m)^{s-\alpha}d_{k_m}^{\alpha-s}(r_mx) \qquad \textrm{in}\quad (r_m^{-1}\Omega_{k_m})\cap B_{r_m^{-1}}.\]
Thus, denoting $\Omega_m=r_m^{-1}\Omega_{k_m}$ and $d_m(x)={\rm dist}(x,r_m^{-1}\Omega_{k_m})$, we have
\begin{equation}\label{vm-eq}
|L_m v_m|\leq \frac{C}{\theta(r_m)}\,d_m^{\alpha-s}(x) \qquad \textrm{in}\quad \Omega_m\cap B_{r_m^{-1}}.
\end{equation}
Notice that the domains $\Omega_m$ converge locally uniformly to $\{x_n>0\}$ as $m\to\infty$.

Next, by Proposition \ref{prop-Cs-2}, we find that for each fixed $M\geq1$
\[\|v_m\|_{C^s(B_M)}\leq C(M)\qquad\textrm{for all}\ m\ \textrm{with}\ r_m^{-1}>2M.\]
The constant $C(M)$ does not depend on $m$.
Hence, by Arzel\`a-Ascoli theorem, a subsequence of $v_m$ converges locally uniformly to a function $v\in C(\R^n)$.

In addition, there is a subsequence of operators $L_{k_m}$ which converges weakly to some operator $L$ of the form \eqref{operator-L}-\eqref{ellipt-const} (see Lemma 3.1 in \cite{RS-stable}).
Hence, for any fixed $K\subset\subset \{x_n>0\}$, thanks to the growth condition \eqref{vm-growth} and since $v_m\rightarrow v$ locally uniformly, we can pass to the limit the equation \eqref{vm-eq} to get
\[Lv=0\quad\textrm{in}\ K.\]
Here we used that the domains $\Omega_m$ converge uniformly to $\{x_n>0\}$, so that for $m$ large enough we will have $K\subset\Omega_m\cap B_{r_m^{-1}}$.
We also used that, in $K$, the right hand side in \eqref{vm-eq} converges uniformly to $0$.

Since this can be done for any $K\subset\subset \{x_n>0\}$, we find
\[Lv=0\quad\textrm{in}\ \{x_n>0\}.\]
Moreover, we also have $v=0$ in $\{x_n\leq 0\}$, and $v\in C(\R^n)$.

Thus, by the classification result \cite[Theorem 4.1]{RS-stable}, we find
\begin{equation}\label{v-eq-final}
v(x)=\kappa(x_n)_+^s
\end{equation}
for some $\kappa\in\R$.

Now, notice that, up to a subsequence, $r_m^{-1}\psi_{k_m}(r_mx)\rightarrow c_1(x_n)_+$ uniformly, with $c_1>0$.
This follows from the fact that $\psi_{k_m}$ are $C^{1,\alpha}(\overline\Omega_{k_m})$ (uniformly in $m$) and that $0<c_0d_{k_m}\leq \psi_{k_m}\leq C_0d_{k_m}$.

Then, multiplying \eqref{vm-orth} by $(r_m)^{-s}$ and passing to the limit, we find
\[\int_{B_1}v(x)(x_n)_+^sdx=0.\]
This means that $\kappa=0$ in \eqref{v-eq-final}, and therefore $v\equiv0$.
Finally, passing to the limit \eqref{vm-nondeg} we find a contradiction, and thus the proposition is proved.
\end{proof}

We finally give the:

\begin{proof}[Proof of Theorem \ref{thm-bdry-reg}]
First, dividing by a constant if necessary, we may assume
\[\|f\|_{L^\infty(B_1\cap\Omega)}+\|u\|_{L^\infty(\R^n)}\leq 1.\]
Second, by definition of $\psi$ we have $\psi/d \in C^{\alpha}(\overline\Omega\cap B_{1/2})$ and
\[\|\psi^s/d^s\|_{C^\alpha(\overline\Omega\cap B_{1/2})}\leq C.\]
Thus, it suffices to show that
\begin{equation}\label{tururu1}
\|u/\psi^s\|_{C^\alpha(\overline\Omega\cap B_{1/2})}\leq C.
\end{equation}

To prove \eqref{tururu1}, let $x_0\in \Omega\cap B_{1/2}$ and $2r=d(x_0)$.
Then, by Proposition \ref{prop-bdry-reg} there is $Q=Q(x_0)$ such that
\begin{equation}\label{tururu2}
\|u-Q\psi^s\|_{L^\infty(B_r(x_0))}\leq Cr^{s+\alpha}.
\end{equation}
Moreover, by rescaling and using interior estimates, we get
\begin{equation}\label{tururu3}
\|u-Q\psi^s\|_{C^\alpha(B_r(x_0))}\leq Cr^s.
\end{equation}
Finally, \eqref{tururu2}-\eqref{tururu3} yield \eqref{tururu1}, exactly as in the proof of Theorem 1.2 in \cite{RS-stable}.
\end{proof}

\begin{rem}
Notice that, thanks to Proposition \ref{prop-bdry-reg}, we have that Theorem \ref{thm-bdry-reg} holds for all right hand sides satisfying $|f(x)|\leq Cd^{\alpha-s}$ in $\Omega$.
\end{rem}

\subsection{Equations with bounded measurable coefficients}

We prove now Theorem \ref{thm-bdry-reg-2}.

First, we show the following $C^\alpha$ estimate up to the boundary.

\begin{prop}\label{prop-Cs-2-bdd}
Let $s\in(0,1)$, and $\Omega$ be any bounded $C^{1,\alpha}$ domain.

Let $u$ be a solution to
\begin{equation}\label{pb-3}
\left\{ \begin{array}{rcll}
M^+u &\geq&-K_0d^{\epsilon-s}&\textrm{in }B_1\cap \Omega \\
M^-u &\leq&K_0d^{\epsilon-s}&\textrm{in }B_1\cap \Omega \\
u&=&0&\textrm{in }B_1\setminus\Omega.
\end{array}\right.\end{equation}
Then,
\[\|u\|_{C^{\bar\alpha}(B_{1/2})}\leq C\left(K_0+\sup_{R\geq1}R^{\delta-2s}\|u\|_{L^\infty(B_R)}\right).\]
The constant $C$ depends only on $n$, $s$, $\epsilon$, $\delta$, $\Omega$, and ellipticity constants.
\end{prop}

\begin{proof}
The proof is very similar to that of Proposition \ref{prop-Cs-2-bdd}.

First, using the supersolution given by Lemma \ref{supersol}, and by the exact same argument of Proposition \ref{prop-Cs-2-bdd}, we find
\[|w|\leq Cd^s\qquad \textrm{in}\quad B_{1/2}.\]
Now, using the interior estimates of \cite{CS} one finds
\[[w]_{C^{\bar\alpha}(B_r(x_0))}\leq C\]
for all balls $B_r(x_0)\subset \Omega\cap B_{1/2}$ with $2r=d(x_0)$, and this yields
\[\|w\|_{C^{\bar\alpha}(B_{1/2})}\leq C,\]
as desired.
\end{proof}

We next show:

\begin{prop}\label{prop-bdry-reg-bdd}
Let $s\in(0,1)$ and $\alpha\in(0,\bar\alpha)$.
Let $L$ be any operator of the form \eqref{operator-L}-\eqref{ellipt-const}, $\Omega$ be any $C^{1,\alpha}$ domain, and $\psi$ be given by Definition \ref{d-is-psi}.

Assume that $0\in \partial\Omega$, and that $\partial\Omega\cap B_1$ can be represented as the graph of a $C^{1,\alpha}$ function with norm less or equal than~1.

Let $u$ be any solution to \eqref{pb-2}, and let
\[K_0=\|f\|_{L^\infty(B_1\cap\Omega)}+\|u\|_{L^\infty(\R^n)}.\]
Then, there exists a constant $Q$ satisfying $|Q|\leq CK_0$ and
\[\bigl|u(x)-Q\psi^s(x)\bigr|\leq CK_0|x|^{s+\alpha}.\]
The constant $C$ depends only on $n$, $s$, and ellipticity constants.
\end{prop}

\begin{proof}
The proof is very similar to that of Proposition \ref{prop-bdry-reg}.

Assume by contradiction that we have $\Omega_k$ and $u_k$ such that:
\begin{itemize}
\item $\Omega_k$ is a $C^{1,\alpha}$ domain that can be represented as the graph of a $C^{1,\alpha}$ function with norm is less or equal than 1;
\item $0\in \partial\Omega_k$ and the unit normal vector to $\partial\Omega_k$ at $0$ is $-e_n$;
\item $u_k$ satisfies \eqref{pb-2} with $K_0=1$;
\item For any constant $Q$, $\sup_{r>0}\sup_{B_r}r^{-s-\alpha}|u_k-Q\psi_k^s|=\infty$.
\end{itemize}
Then, by Lemma \ref{lem-blow-up-seq} we will have
\[\sup_k\sup_{r>0}\|u_k-\phi_{k,r}\|_{L^\infty(B_r)}=\infty,\]
where
\[\phi_{k,r}(x)=Q_k(r)\psi_k^s,\qquad Q_k(r)=\frac{\int_{B_r}u_k\psi_k^s}{\int_{B_r}\psi_k^{2s}}.\]

We now define $\theta(r)$, $r_m\to0$, and $v_m$ as in the proof of Proposition \ref{prop-bdry-reg}.
Then, we have
\begin{equation}\label{vm-orth-bdd}
\int_{B_1} v_m(x)\psi_k^s(r_mx)dx=0,
\end{equation}
\begin{equation}\label{vm-nondeg-bdd}
\|v_m\|_{L^\infty(B_1)}\geq\frac12,
\end{equation}
and
\begin{equation}\label{vm-growth-bdd}
\|v_m\|_{L^\infty(B_R)}\leq CR^{s+\alpha}\qquad\textrm{for all}\quad R\geq1.
\end{equation}

Moreover, the functions $v_m$ satisfy
\[M^- v_m(x)\leq \frac{(r_m)^{2s}}{(r_m)^{s+\alpha}\theta(r_m)}+\frac{(r_m)^{2s}}{(r_m)^{s+\alpha}\theta(r_m)}\,(M^+\psi_{k_m})(r_m x)\]
in $(r_m^{-1}\Omega_{k_m})\cap B_{r_m^{-1}}$.
Using Lemma \ref{lem11}, and denoting $\Omega_m=r_m^{-1}\Omega_{k_m}$ and $d_m(x)={\rm dist}(x,r_m^{-1}\Omega_{k_m})$, we find
\begin{equation}\label{vm-eq-bdd}
M^- v_m\leq \frac{C}{\theta(r_m)}\,d_m^{\alpha-s}(x) \qquad \textrm{in}\quad \Omega_m\cap B_{r_m^{-1}}.
\end{equation}
Similarly, we find
\[M^+ v_m\geq -\frac{C}{\theta(r_m)}\,d_m^{\alpha-s}(x) \qquad \textrm{in}\quad \Omega_m\cap B_{r_m^{-1}}.\]
Notice that the domains $\Omega_m$ converge locally uniformly to $\{x_n>0\}$ as $m\to\infty$.

Next, by Proposition \ref{prop-Cs-2-bdd}, we find that for each fixed $M\geq1$
\[\|v_m\|_{C^{\bar\alpha}(B_M)}\leq C(M)\qquad\textrm{for all}\ m\ \textrm{with}\ r_m^{-1}>2M.\]
The constant $C(M)$ does not depend on $m$.
Hence, by Arzel\`a-Ascoli theorem, a subsequence of $v_m$ converges locally uniformly to a function $v\in C(\R^n)$.

Hence, passing to the limit the equation \eqref{vm-eq-bdd} we get
\[M^-v\leq 0\leq M^+v\quad\textrm{in}\ \{x_n>0\}.\]
Moreover, we also have $v=0$ in $\{x_n\leq 0\}$, and $v\in C(\R^n)$.

Thus, by the classification result \cite[Proposition 5.1]{RS-K}, we find
\begin{equation}\label{v-eq-final-bdd}
v(x)=\kappa(x_n)_+^s
\end{equation}
for some $\kappa\in\R$.
But passing \eqref{vm-orth-bdd} ---multiplied by $(r_m)^{-s}$--- to the limit, we find
\[\int_{B_1}v(x)(x_n)_+^sdx=0.\]
This means that $v\equiv0$, a contradiction with \eqref{vm-nondeg-bdd}.
\end{proof}

Finally, we give the:

\begin{proof}[Proof of Theorem \ref{thm-bdry-reg-2}]
The result follows from Proposition \ref{prop-bdry-reg-bdd}; see the proof of Theorem \ref{thm-bdry-reg}.
\end{proof}

\section{Barriers: $C^1$ domains}
\label{sec-barriers-C1}

We construct now sub and supersolutions that will be needed in the proof of Theorem \ref{thmC1}.
Recall that in $C^1$ domains one does not expect solutions to be comparable to $d^s$, and this is why the sub and supersolutions we construct have slightly different behaviors near the boundary.
Namely, they will be comparable to $d^{s+\epsilon}$ and $d^{s-\epsilon}$, respectively.

\begin{lem}\label{homog-subsol}
Let $s\in(0,1)$, and $e\in S^{n-1}$.
Define
\[\Phi_{\rm sub}(x) := \left( e\cdot x- \eta |x| \left(1- \frac{(e\cdot x)^2}{|x|^2} \right)\right)_+^{s+\epsilon}\]
and
\[\Phi_{\rm super}(x) := \left( e\cdot x + \eta |x| \left(1- \frac{(e\cdot x)^2}{|x|^2} \right)\right)_+^{s-\epsilon}\]
For every $\epsilon>0$ there is $\eta>0$ such that two functions $\Phi_{\rm sub}$ and $\Phi_{\rm super}$ satisfy, for all $L\in \mathcal L_{*}$,
\[\begin{cases}
L \Phi_{\rm sub} \ge c_\epsilon d^{\epsilon-s}>0 \quad & \mbox{in }\mathcal C_\eta  \\
\Phi_{\rm sub} = 0  \quad & \mbox{in }\R^n \setminus \mathcal C_\eta\\
\end{cases}\]
and
\[\begin{cases}
L \Phi_{\rm super} \le -c_\epsilon d^{-\epsilon-s}<0 \quad & \mbox{in }\mathcal C_{-\eta}  \\
\Phi_{\rm super} = 0  \quad & \mbox{in }\R^n \setminus \mathcal C_{-\eta}\\
\end{cases}\]
where $\mathcal C_{\pm\eta}$ is the cone
\[\mathcal C_{\pm\eta}: = \left\{ x \in \R^n\ :  e\cdot \frac{x}{|x|}    >  \pm \eta \left( 1 -  \left( e\cdot \frac{ x }{|x|} \right)^2\right) \right\}.\]
The constant $\eta$ depends only on $\epsilon$, $s$, and ellipticity constants.
\end{lem}

\begin{proof}
We prove the statement for $\Phi_{\rm sub}$.
The statement for $\Phi_{\rm super}$ is proved similarly.

Let us denote $\Phi := \Phi_{\rm sub}$.
By homogeneity it is enough to prove that $L\Phi \ge c_\epsilon>0$ on points belonging to $e + \partial \mathcal{  C}_\eta$, since all the positive dilations of this set with respect to the origin cover the interior of $\mathcal{\tilde  C}_\eta$.

Let thus $P\in  \partial \mathcal{ C}_\eta$, that is,
\[ e\cdot P- \eta \left( |P| -  \frac{(e\cdot P)^2}{|P|} \right) =0.\]

Consider
\[\begin{split}
\Phi_{P,\eta}(x) &:= \Phi(P+e+x)
 \\
&= \left( e\cdot (P+e+x)- \eta \left( |P+e+x| -  \frac{(e\cdot (P+e+x))^2}{|P+e+x|} \right)\right)_+^{s+\epsilon}
\\
&= \left( 1 +e\cdot x- \eta \left( |P+e+x| -|P|-  \frac{(e\cdot (P+e+x))^2}{|P+e+x|} +\frac{(e\cdot P)^2}{|P|} \right)\right)_+^{s+\epsilon}
\\
&=\bigl( 1 +e\cdot x- \eta \psi_P(x)  \bigr)_+^{s+\epsilon},
\end{split}\]
where we define
\[\psi_P(x)  :=  |P+e+x| -|P|-  \frac{(e\cdot (P+e+x))^2}{|P+e+x|} +\frac{(e\cdot P)^2}{|P|} .\]

Note that the functions $\psi_P$  satisfy
\[\psi_P(0) = 0,\]
\[ |\nabla \psi_P(x)|  \le  C  \quad \mbox{in } \R^n \setminus \{ -P-e\},  \]
and
\[ |D^2 \psi_P(x)|  \le  C  \quad \mbox{for  }x \in B_{1/2}, \]
where $C$ does not depend on $P$ (recall that $|e|=1$).

Then, the family $\Phi_{P,\eta}$ satisfies
\[ \Phi_{P,\eta} \rightarrow  (1+e\cdot x)_+^{s+\epsilon} \quad \mbox{in }C^2(\overline{B_{1/2}})\]
as $\eta\searrow 0$, uniformly in $P$ and moreover
\[ \int_{\R^n}  \frac{ \bigl|\Phi_{P,\eta}- (1+e\cdot x)_+^{s+\epsilon}  \bigr|}{ 1+|x|^{n+2s} } \,dx  \le  \int_{\R^n}  \frac{ C (C \eta |x|) ^{s+\epsilon} }{ 1+|x|^{n+2s} } \,dx \le C\eta^{s+\epsilon}.\]
Thus,
\[ L\Phi_{P,\eta}(0)  \rightarrow L\bigl((1+e\cdot \quad )_+^{s+\epsilon}\bigr)(0) \geq c(s,\epsilon,\lambda) >0 \quad \mbox{as }\eta\searrow 0\]
uniformly in $P$.

In particular one can chose $\eta = \eta(s,\epsilon,\lambda,\Lambda)$ so that $L\Phi_{P,\eta}(0) \ge c_\epsilon>0$ for all $P\in \partial \mathcal {\tilde C}_\eta$ and for all $L\in \mathcal L_*$, and the lemma is proved.
\end{proof}

\section{Regularity in $C^1$ domains}
\label{sec-regularity-C1}

We prove here Theorems \ref{thmC1} and \ref{thmC1-2}.

\begin{defi}\label{def-C1-at-0}
Let $r_0>0$ and let $\rho:(0,r_0]\rightarrow 0$ be a nonincreasing function with $\lim_{t\downarrow 0}\rho(t) =0$.
We say that a domain $\Omega$ is {\em improving Lipschitz} at $0$ with inwards unit normal vector $e_n = (0,\dots,0,1)$ and modulus $\rho$ if
\[ \Omega\cap B_r = \{(x',x_n)\,:\,x_n > g(x')\}\cap B_r\qquad \textrm{for}\quad r\in(0,r_0],\]
where $g: \R^{n-1}\rightarrow \R$ satisfies
\[ \|g\|_{{\rm Lip}(B_r)}\le \rho(r)\quad \mbox{for }0<r\le r_0.\]
We say that $\Omega$ is {\em improving Lipschitz} at $x_0\in\partial\Omega$ with inwards unit normal $e\in S^{n-1}$ if the normal vector to $\partial\Omega$ at $x_0$ is $e$ and, after a rotation, the domain $\Omega-x_0$ satisfies the previous definition.
\end{defi}

We first prove the following $C^\alpha$ estimate up to the boundary.

\begin{lem}\label{C-alpha-est}
Let $s\in(0,1)$, and let $\Omega \subset \R^n$ be a Lipschitz domain in $B_1$ with Lipschitz constant less than $\ell$.
Namely, assume that after a rotation we have
\[ \Omega\cap B_1 = \{(x',x_n)\,:\,x_n > g(x')\}\cap B_1,\]
with $\|g\|_{{\rm Lip}(B_1)}\le \ell$.
Let $u\in C(B_1)$ be a viscosity solution of
\[ M^+ u \ge -K_0d^{\epsilon-s} \quad \mbox{and}\quad M^- u \le K_0d^{\epsilon-s} \quad \mbox{in }\Omega\cap B_1,\]
\[ u= 0 \quad \mbox{in }B_1\setminus\Omega.\]
Assume that
\[\|u\|_{L^\infty(B_R)}\le K_0R^{2s-\epsilon} \quad \mbox{for all } R\ge1.\]
Then, if $\ell\le \ell_0$, where $\ell_0= \ell_0(n,s, \lambda, \Lambda)$, we have
\[\|u\|_{C^{\bar\alpha}(B_{1/2})}\leq CK_0.\]
The constants $C$ and $\bar\alpha$ depend only on $n$, $s$, $\epsilon$ and ellipticity constants.
\end{lem}

\begin{proof}
By truncating $u$ in $B_2$ and dividing it by $CK_0$ we may assume that
\[\|u\|_{L^\infty(\R^n)}=1\]
and that
\[ M^+ u \ge -d^{\epsilon-s} \quad \mbox{and}\quad M^- u \le d^{\epsilon-s} \quad \mbox{in }\Omega\cap B_1.\]
Now, we divide the proof into two steps.

{\em Step 1.} We first prove that
\begin{equation}\label{Calphaest1}
\bigl|u(x)\bigr|\le C|x-x_0|^\alpha\quad\textrm{in}\ \Omega\cap B_{3/4},
\end{equation}
where $x_0\in\partial\Omega$ is the closest point to $x$ on $\partial\Omega$.
We will prove \eqref{Calphaest1} by using a supersolution.
Indeed, given $\epsilon\in(0,s)$, let $\Phi_{\rm super}$ and $\mathcal C_\eta$ be the homogeneous supersolution and the cone from Lemma \ref{homog-subsol}, where $e = e_n$.
Note that $\Phi_{\rm super}$ is a positive function satisfying $M^- \Phi_{\rm super}\le -cd^{-\epsilon-s}<0$ outside the convex cone $\R^n\setminus \mathcal C_\eta$, and it is homogeneous of degree $s-\epsilon$.

Then, we easily check that the function $\psi = C\Phi_{\rm super}-\chi_{B_1(z_0)}$, with $C$ large and $|z_0|\geq2$ such that $\Phi_{\rm super}\geq1$ in $B_1(z_0)$, satisfies $M^+ \psi \le -d^{\epsilon-s}$ in $B_{1/4}\cap\mathcal C_\eta$ and $\psi\geq\frac14$ in $\mathcal C_\eta\setminus B_{1/4}$.
Indeed, we simply use that $M^-\chi_{B_1(z_0)}\geq c_0>0$ in $B_{1/4}$.
Note that this argument exploits the nonlocal character of the operator and a slightly more complicated one would be needed in order to obtain a result that is stable as $s\uparrow 1$.

Note that the supersolution $\psi$ vanishes in $B_{1/4}\setminus \mathcal C_\eta$.
Then, if $\ell_0$ is small enough, for  every point in $x_0\in\partial\Omega \cap B_{3/4}$ we will have
\[x_0+(B_{1/4}\setminus \mathcal C_\eta)\subset B_1\setminus\Omega.\]

Then, using translates of $\psi$  (and $-\psi$) upper (lower) barriers we get $\bigl|u(x)\bigr|\le \psi(x_0+x)\leq C|x-x_0|^{s-\epsilon}$, as desired.

{\em Step 2.}  To obtain a $C^\alpha$ estimate up to the boundary, we use the following interior estimate from \cite{CS}: Let $r\in(0,1)$,
\[ M^+ u\ge -r^{\alpha-2s} \quad \mbox{and } M^- u\le r^{\alpha-2s} \quad \mbox{in }B_r(x)\]
and
\[ |u(z)|\le r^\alpha\left( 1+ \frac{(z-x)^\alpha}{r^\alpha} \right)\quad \mbox{in all of }\R^n.\]
Then,
\[  [u]_{C^\alpha(B_{r/2}(x))} \le C,\]
with $C$ and $\alpha>0$ depending only $s$, ellipticity constants and dimension.

Combining this estimate with \eqref{Calphaest1}, it follows that
\[\|u\|_{C^\alpha(B_{1/2})}\leq C.\]
Thus, the lemma is proved.
\end{proof}

We will also need the following.

\begin{lem}\label{lemvarphi}
Let $s\in(0,1)$, $\alpha\in (0,\bar\alpha)$, and $C_0\ge 1$.
Given $\epsilon\in (0,\alpha]$, there exist $\delta>0$ depending only on $\epsilon$, $n$, $s$, and ellipticity constants, such that the following statement holds.

Assume that $\Omega \subset \R^n$ is a Lipchitz domain such that $\partial\Omega\cap B_{1/\delta}$ is a Lipchitz graph of the form $x_n = g(x')$, in $|x'|< 1/\delta$ with
\[ [g]_{{\rm Lip}(B_{1/\delta})} \le \delta,\]
and  $0\in \partial\Omega$.

Let $\varphi\in C(\R^n)$ be a viscosity solution of
\[ M^+ \varphi \ge -\delta\,d^{\epsilon-s} \quad \mbox{and}\quad M^- \varphi \le \delta\,d^{\epsilon-s} \quad \mbox{in }\Omega\cap B_{1/\delta},\]
\[ \varphi= 0 \quad \mbox{in } B_{1/\delta}\setminus\Omega,\]
satisfying
\[ \varphi\ge 0  \quad \mbox{in }B_1.\]

Assume that $\varphi$ satisfies
\[ \sup_{B_1} \varphi =1 \quad \mbox{and}\quad  \|\varphi\|_{L^\infty(B_{2^l})}\le C_0 (2^l)^{s+\alpha}\quad \mbox{for all } l\ge0.\]
Then, we have
\begin{equation}\label{integral}
\int_{B_1} \varphi^2\,dx \ge  \frac 1 2\int_{B_1} (x_n)_+^{2s}\,dx
\end{equation}
and
\begin{equation}  \label{iterated}
\left(\frac{1}{2}\right)^{s+\epsilon} \le  \frac{\sup_{B_{2^{l-1}}}\varphi }{\sup_{B_{2^l}}\varphi}  \le \left(\frac{1}{2}\right)^{s-\epsilon}
\quad \mbox{for all }l\le 0.
\end{equation}
\end{lem}

\begin{proof}
{\em Step 1.} We first prove that, for $\delta$ small enough, we have \eqref{integral} and
\begin{equation}  \label{step}
\left(\frac{1}{2}\right)^{s+\epsilon} \le  \frac{\sup_{B_{1/2}}\varphi }{\sup_{B_{1}}\varphi}  \le \left(\frac{1}{2}\right)^{s-\epsilon}
\end{equation}
In a second step we will iterate \eqref{step} to show \eqref{iterated}.

The proof of \eqref{step} is by compactness.
Suppose that there is a sequence $\varphi_k$ of functions satisfying the assumptions with $\delta = \delta_k \downarrow 0$ for which one of the three possibilities
\begin{equation}\label{contradiction1}
 \left(\frac{1}{2}\right)^{s+\epsilon} > \frac{\sup_{B_{1/2}}\varphi_k }{\sup_{B_{1}}\varphi_k},
\end{equation}
\begin{equation}\label{contradiction2}
\frac{\sup_{B_{1/2}}\varphi_k }{\sup_{B_{1}}\varphi_k} > \left(\frac{1}{2}\right)^{s-\epsilon}
\end{equation}
or
\begin{equation}\label{contradiction3}
\int_{B_1} \varphi_k^2\,dx <  \frac 1 2\int_{B_1} (x_n)_+^{2s}\,dx
\end{equation}
holds for all $k\ge 1$.

Let us show that a subsequence of $\varphi_k$ converges locally uniformly $\R^n$ to the function $(x_n)_+^s$.
Indeed, thanks to Lemma \ref{C-alpha-est} and the Arzela-Ascoli theorem a subsequence of $\varphi_k$ converges to a function $\varphi\in C(\R^n)$, which satisfies $M^+\varphi\geq0$ and $M^-\varphi\leq 0$ in $\R^n_+$, and $\varphi=0$ in $\R^n_-$.
Here we used that $\delta_k\to 0$.
Moreover, by the growth control $\|\varphi\|_{L^\infty(B_R)}\leq CR^{s+\alpha}$ and the classification theorem \cite[Proposition 5.1]{RS-K}, we find $\varphi(x)=K(x_n)_+^s$.
But since $\sup_{B_1}\varphi_k=1$, then $K=1$.

Therefore, we have proved that a subsequence of $\varphi_k$ converges uniformly in $B_1$ to $(x_n)_+^s$.
Passing to the limit \eqref{contradiction1}, \eqref{contradiction2} or \eqref{contradiction3}, we reach a contradiction.

\vspace{2mm}

{\em Step 2.} We next show that we can iterate \eqref{step} to obtain \eqref{iterated} by induction.
Assume that for some $m\le 0$ we have
\begin{equation}\label{doubling}
\left(\frac{1}{2}\right)^{s+\epsilon} \le  \frac{\sup_{B_{2^{l-1}}}\varphi }{\sup_{B_{2^l}}\varphi}  \le \left(\frac{1}{2}\right)^{s-\epsilon}\quad \mbox{for  }m\le l\le 0.
\end{equation}
We then consider the function
\[ \bar \varphi = \frac{\varphi(2^{-m} x)}{\sup_{B_{2^m}} \varphi},\]
and notice that
\[2^{(s+\epsilon)l} \le  \sup_{B_{2^l}}\varphi   \le 2^{(s-\epsilon)l}\quad \mbox{for  }m\le l\le 0.\]

Thus,
\[ M^+ \bar \varphi \ge -\frac{\delta 2^{2s m}}{2^{(s+\epsilon)m}} \ge -\delta \quad \mbox{in } (2^{-m}\Omega)\cap B_{2^{-m}/\delta}\]
and similarly
\[  M^- \bar \varphi \le \delta \quad \mbox{in } (2^{-m}\Omega)\cap B_{2^{-m}/\delta}.\]
Clearly
\[ \bar \varphi = 0 \quad \mbox{in } (2^{-m}\mathcal C\Omega)\cap B_{2^{-m}/\delta} \]
and
\[\varphi\ge0 \quad \mbox{in } B_{2^{-m}}\supset B_1.\]

Since $\partial \Omega$ is Lipchitz with constant $\delta$ in $B_{1/\delta}$ and $2^{-m}\ge1$ ($m\le 0$) we have that the rescaled domain $(2^{-m}\Omega)\cap B_{2^{-m}/\delta}$ is also Lipchitz with the same constant $1/\delta$ in a larger ball.

Finally, using \eqref{doubling} again we find
\[\sup_{B_{2^l}} \bar \varphi = \frac{\sup_{B_{2^{l+m}}}  \varphi}{\sup_{B_{2^{m}}} \varphi} \le 2^{(s+\epsilon)l}  \le 2^{(s+\alpha)l}\qquad \textrm{for}\ l\ge0\ \textrm{with}\ l+m\le 0,\]
For $l+m>0$ we have
\[\sup_{B_{2^l}} \bar \varphi = \frac{\sup_{B_{2^{l+m}}}  \varphi}{2^{(s+\epsilon)m}  \varphi} \le \frac{C_02^{(s+\alpha)(l+m)}}{2^{(s+\epsilon)m}} = C_02^{(s+\alpha)l} 2^{(\alpha-\epsilon)m} \le C_02^{(s+\alpha)l}.\]

Hence, using Step 1, we obtain
\[ \left(\frac{1}{2}\right)^{s+\epsilon} \le  \frac{\sup_{B_{1/2}}\bar \varphi }{\sup_{B_{1}}\bar \varphi}  \le \left(\frac{1}{2}\right)^{s-\epsilon}.\]
Thus \eqref{doubling} holds for $l= m-1$, and the lemma is proved.
\end{proof}

We next prove the following.

\begin{prop}\label{bdryregC1doms}
Let $s\in(0,1)$, $\alpha\in (0,\bar\alpha)$, and $C_0\ge 1$.

Let $\Omega\subset \R^n$ be a domain that is improving Lipschitz at $0$ with unit outward normal $e\in S^{n-1}$ and with modulus of continuity $\rho$ (see Definition \ref{def-C1-at-0}).
Then, there exists $\delta>0$, depending only on $\alpha$, $s$, $C_0$, ellipticity constants, and dimension such that the following statement holds.

Assume that $r_0=1/\delta$ and $\rho(1/\delta)<\delta$.
Suppose that $u,\varphi \in C(\R^n)$ are viscosity solutions of
\begin{equation}\label{eqnpointbdryreg}
\begin{cases}
M^+ (au+b\varphi)  \ge  -\delta(|a|+|b|)d^{\alpha-s}  \quad &\mbox{in } B_{1/\delta}\cap \Omega\\
u=\varphi =  0 & \mbox{in }B_{1/\delta}\setminus \Omega,
\end{cases}
\end{equation}
for all $a,b\in \R$.
Moreover, assume that
\begin{equation}\label{assumptiongrowth}
\|au+b\varphi\|_{L^\infty(\R^n)}\le C_0 (|a|+|b|)R^{s+\alpha} \quad \mbox{for all }R\ge 1,
\end{equation}
\[ \varphi\ge 0 \quad \mbox{in }B_{1}, \quad\mbox{and}\quad  \sup_{B_1}\varphi =1.\]

Then, there is $K\in \R$ with $|K|\le C$ such that
\[ \bigl| u(x) - K\varphi(x) \bigr| \le C|x|^{s+\alpha}\quad\mbox{in }B_1, \]
where $C$ depends only on $\rho$, $C_0$, $\alpha$, $s$, ellipticity constants, and dimension.
\end{prop}

\begin{proof}
{\em Step 1 (preliminary results).} Fix $\epsilon\in(0,\alpha)$. Using Lemma~\ref{lemvarphi}, if $\delta$ is small enough we have
\begin{equation}\label{integral2}
\int_{B_1} \varphi^2\,dx \ge  \frac 1 2\int_{B_1} (x_n)_+^{2s}\,dx \ge c(n,s)>0
\end{equation}
and
\begin{equation}  \label{iterated2}
\left(\frac{1}{2}\right)^{s+\epsilon} \le  \frac{\sup_{B_{2^{l-1}}}\varphi }{\sup_{B_{2^l}}\varphi}  \le \left(\frac{1}{2}\right)^{s-\epsilon}\quad \mbox{for all }l\le 0.
\end{equation}

In particular, since $\sup_{B_{1}}\varphi=1$ then
\begin{equation}\label{iteratedwithr}
 (r/2)^{s+\epsilon} \le  \sup_{B_{r}} \varphi  \le (2r)^{s-\epsilon}\quad \mbox{for all } r\in(0,1).
\end{equation}

\vspace{2mm}

{\em Step 2.}  We prove now, with a blow-up argument, that
\begin{equation}\label{goal}
\bigl\| u(x) - K_r \varphi(x) \bigr\|_{L^\infty(B_r)} \le C r^{s+\alpha}
\end{equation}
for all $r\in(0,1]$, where
\begin{equation}\label{Kr}
 K_r := \frac{\int_{B_r} u\,\varphi\,dx}{\int_{B_r} \varphi^2\,dx}.
\end{equation}

Notice that \eqref{goal} implies the estimate of the proposition with $K = \lim_{r\searrow 0}K_r$. Indeed, we have $|K_1|\le C$ ---which is immediate using \eqref{assumptiongrowth} with $a=1$ and $b=0$ and \eqref{integral2}--- and
\[\begin{split}
|K_{r}-K_{r/2}|(r/2)^{s+\epsilon}  &\le \bigl\|K_{r}\varphi-K_{r/2}\varphi\bigr\|_{L^\infty(B_r)}
\\
&\le \bigl\| u - K_r \varphi\bigr\|_{L^\infty(B_r)}  + \bigl\| u - K_{r/2} \varphi \bigr\|_{L^\infty(B_r)}
\\
&\le Cr^{s+\alpha}.
\end{split}\]
Thus,
\[ |K| \le |K_1| + \sum_{j=0}^\infty |K_{2^{-j}}-K_{2^{-j-1}}| \le C + C\sum_{j=0}^\infty 2^{-j(\alpha-\epsilon)}\leq C,\]
provided that $\epsilon$ is taken smaller that $\alpha$.

Let us prove \eqref{goal} by contradiction.
Assume that we have a sequences $\Omega_j$, $e_j$,$u_j$, $\varphi_j$ satisfying the assumptions of the Proposition, but not \eqref{goal}.
That is,
\[ \lim_{j\to\infty} \sup_{r>0} r^{-s-\alpha}\bigl\| u_j(x) - K_{r,j} \varphi_j \bigr\|_{L^\infty(B_{r})}= \infty,\]
were $K_{r,j}$ is defined as in \eqref{Kr} with $u$ replaced by $u_j$ and $\varphi$ replace by $\varphi_j$.

Define, for $r\in(0,1]$ the nonincreasing quantity
\[ \theta(r) = \sup_{r'\in (r,1)} (r')^{-s-\alpha} \bigl\| u_j(x) - K_{r',j} \varphi_j \bigr\|_{L^\infty(B_{r'})}.\]
Note that $\theta(r)<\infty$ for $r>0$ since  $\|u_j\|_{L^\infty(\R^n)}\le 1$ and that $\lim_{r\searrow 0} \theta(r)=\infty$.

For every $m \in \mathbb N$, by definition of $\theta$ there exist $r'_m\ge1/m$, $j_m$, $\Omega_m= \Omega_{j_m}$, and  $e_m= e_{j_m}$ such that
\[ (r'_m)^{-s-\alpha} \bigl\| u_{j_m}(x) - K_{r'_m,j_m}  \varphi_{j_m} \bigr \|_{L^\infty(B_{r'_m})} \ge \frac 12\theta(1/m) \ge \frac 12 \theta(r'_m).\]
Note that $r'_m \rightarrow 0$.
Taking a subsequence we may assume that $e_m \to e \in S^{n-1}$. Denote
\[u_m = u_{j_m},\quad K_m = K_{r'_m,j_m}\quad\mbox{and}\quad  \varphi_m = \varphi_{j_m}.\]

We now consider the blow-up sequence
\[ v_m(x) = \frac{ u_m(r'_m x)-   K_m \varphi_m(r'_m x) }{ (r'_m)^{s+\alpha} \theta(r'_m)}.\]
By definition of $\theta$ and $r'_m$ we will have
\begin{equation}\label{nondeg0}
\|v_m\|_{L^\infty(B_1)} \ge \frac 12.
\end{equation}
In addition, by definition of $K_m = K_{r'_m,j_m}$ we have
\begin{equation}\label{orthog}
\int_{B_1}  v_m(x)\varphi_m(r'_m x) \,dx=0
\end{equation}
for all $m\geq1$.

Let us prove that
\begin{equation}\label{growthctrl}
\|v_m\|_{L^\infty(B_R)} \le  C R^{s+\alpha}\quad \mbox{for all } R\ge 1.
\end{equation}

Indeed, first, by definition of $\theta(2r)$ and $\theta(r)$,
\[
\begin{split}
\frac{\bigl\|K_{2r,j}\varphi_j-K_{r,j} \varphi_j \bigr\|_{L^\infty(B_{r})}}{r^{s+\alpha}\theta(r)}
&\le \frac{2^{s+\alpha}\theta(2r)}{\theta(r)}\,
\frac{ \bigl\|u_j- K_{r,j}\varphi_j \bigr\|_{L^\infty(B_{2r})}}{(2r)^{s+\alpha}\theta(2r)}
+ \frac{\bigl\|u_j-K_{r/2,j} \varphi_j \bigr\|_{L^\infty(B_{r})}}{r^{s+\alpha}\theta(r)}
\\
&\le 2^{s+\alpha}+1 \le 5.
\end{split}
\]

On the one hand, using Step 1 we have
\[\begin{split}
\frac{|K_{2r,j}-K_{r,j}| (r/2)^{s+\epsilon}}{r^{s+\alpha}\theta(r)}
&\le  \frac{|K_{2r,j}-K_{r,j}| \,\|\varphi_j \|_{L^\infty(B_{r})}}{r^{s+\alpha}\theta(r)}
\\
& = \frac{\bigl\|K_{2r,j}\varphi_j-K_{r,j} \varphi_j \bigr\|_{L^\infty(B_{r})}}{r^{s+\alpha}\theta(r)}
\\
&\le 5,
\end{split}\]
and therefore
\begin{equation}\label{difK's}
|K_{2r,j}-K_{r,j}|  \le 10\, r^{\alpha-\epsilon}\theta(r),
\end{equation}
which we will use later on in this proof.

On the other hand, by \eqref{iterated2} in Step 1 we have, whenever $0<2^lr\le 2^Nr\le 1$,
\[\bigl\| \varphi_j \bigr\|_{L^\infty(B_{2^N r})} \le 2^{(s+\epsilon)(N-l)}\bigl\| \varphi_j \bigr\|_{L^\infty(B_{2^l r})} \]
and therefore
\[\begin{split}
\frac{\bigl\|K_{2^{l+1}r,j}\varphi_j-K_{2^l r,j} \varphi_j \bigr\|_{L^\infty(B_{2^N r})}}{r^{s+\alpha}\theta(r)}
&= \frac{ \bigl|K_{2^{l+1}r,j}-K_{2^l r,j}\bigr|\,\bigl\| \varphi_j \bigr\|_{L^\infty(B_{2^N r})}}{r^{s+\alpha}\theta(r)}
\\
&\le \frac{ \bigl|K_{2^{l+1}r,j}-K_{2^l r,j}\bigr|\, 2^{(s+\epsilon)(N-l)}\bigl\| \varphi_j \bigr\|_{L^\infty(B_{2^l r})}}{r^{s+\alpha}\theta(r)}
\\
&=  \frac{2^{(s+\epsilon)(N-l)} \bigl\|K_{2^{l+1}r,j}\varphi_j-K_{2^l r,j} \varphi_j \bigr\|_{L^\infty(B_{2^l r})}}{r^{s+\alpha}\theta(r)}
\\
&=  \frac{2^{l(s+\alpha)} \theta(2^lr)}{\theta(r)} \frac{2^{(s+\epsilon)(N-l)} \bigl\|K_{2^{l+1}r,j}\varphi_j-K_{2^l r,j} \varphi_j \bigr\|_{L^\infty(B_{2^l r})}}{(2^lr)^{s+\alpha}\theta(2^lr)}
\\
&\le 10 \,2^{(s+\epsilon)N}\,2^{l(\alpha-\epsilon)}.
\end{split}\]
Thus,
\[\frac{\bigl\|K_{2^{N}r,j}\varphi_j-K_{r,j} \varphi_j \bigr\|_{L^\infty(B_{2^N r})}}{r^{s+\alpha}\theta(r)}
\le  \,2^{(s+\epsilon)N} \sum_{l=0}^{N-1}2^{l(\alpha-\epsilon)}  \le C  2^{(s+\alpha)N},\]
where we have used that $\epsilon\in(0,\alpha)$.

Form the previous equation we deduce
\[
\frac{\bigl\|K_{Rr,j}\varphi_j-K_{r,j} \varphi_j \bigr\|_{L^\infty(B_{R r})}}{r^{s+\alpha}\theta(r)}
\le C  R^{s+\alpha}
\]
whenever $0<r\le Rr\le 1$.

Hence,
\[
\begin{split}
\|v_m\|_{L^\infty(B_R)} &= \frac{1}{\theta(r'_m)(r'_m)^{s+\alpha}} \bigl\|u_m - K_m\varphi_m \bigr\|_{L^\infty(B_{Rr'm})}
\\
&\le \frac{R^{s+\alpha}\bigl\|u_{j_m} - K_{Rr'_m,j_m}\varphi_{j_m} \bigr\|_{L^\infty(B_{Rr'm})}}{\theta(r'_m)(Rr'_m)^{s+\alpha}}
 + \frac{\bigl\|K_{Rr'_m,j_m}\varphi_{j_m} -K_{r'_m,j_m} \varphi_{j_m} \bigr\|_{L^\infty(B_{{Rr'_m})}}}{(r'_m)^{s+\alpha}\theta(r'_m)}
\\
&\le  \frac{R^{s+\alpha}\theta(Rr'_m)}{\theta(r'_m)} + CR^{s+\alpha}
\\
&\le C R^{s+\alpha},
\end{split}
\]
whenever $Rr'_m \le 1$.

When $Rr'_m \ge 1$ we simply use the assumption \eqref{assumptiongrowth}, namely,
\[ \| au_m+b\varphi_m \|_{L^\infty(\R^n)}\le C_0(|a|+|b|)R^{s+\alpha} \quad \mbox{for all }R\ge 1,\]
twice, with $a=1$, $b= -K_{1,j_m}$ and with $a=0$, $b= 1$  to estimate
\[
\begin{split}
\|v_m\|_{L^\infty(B_R)} &= \frac{1}{\theta(r'_m)(r'_m)^{s+\alpha}} \bigl\|u_m - K_m \varphi_m \bigr\|_{L^\infty(B_{Rr'm})}
\\
&\le \frac{R^{s+\alpha}\bigl\|u_{j_m} - K_{1,j_m}\varphi_{j_m} \bigr\|_{L^\infty(B_{Rr'm})}}{\theta(r'_m)(Rr'_m)^{s+\alpha}}
 + \frac{ \bigl\|K_{1,j_m}\varphi_{j_m}-K_{r'_m,j_m} \varphi_{j_m} \|_{L^\infty(B_{R r'_m})}}{(r'_m)^{s+\alpha}\theta(r'_m)}
\\
&\le C_0 (1+|K_{1,j_m}| )  R^{s+\alpha}
+ \frac{ \bigl\|K_{1,j_m}\varphi_{j_m}-K_{r'_m,j_m} \varphi_{j_m} \|_{L^\infty(B_{1})}}{(r'_m)^{s+\alpha}\theta(r'_m)} \frac{\|\varphi_{j_m} \|_{L^\infty(B_{R r'_m})}}{\|\varphi_{j_m} \|_{L^\infty(B_{1})}}
\\
&\le C R^{s+\alpha} + C\left( \frac{1}{r'_m}\right)^{s+\alpha} (Rr'_m)^{s+\epsilon}
\\
&\le C R^{s+\alpha} + C(r'_m)^{-s-\alpha} (Rr'_m)^{s+\alpha} \le CR^{s+\alpha},
\end{split}
\]
where we have used $|K_{1,j_m}|\le C$ (that we will prove in detail in Step 3).

\vspace{2mm}

{\em Step 3.}  We prove that a subsequence of $v_m$ converges locally uniformly to a entire solution $v_\infty$ of the problem
\begin{equation}\label{limitting}
\begin{cases}
M^+ v_\infty \ge  0 \ge M^- v_\infty \quad &\mbox{in } \{e\cdot x >0\}\\
v_\infty =  0 &\mbox{in } \{e\cdot x <0\}.
\end{cases}
\end{equation}

By assumption,  the function $w = au_m + b\varphi_m$ satisfies
\begin{equation}\label{eqnpointbdryreg2}
\begin{cases}
M^+ (au_m+b\varphi_m)  \ge  -\delta(|a|+|b|)d^{\alpha-s}  \quad &\mbox{in } B_1\cap \Omega_m\\
u_m=\varphi_m =  0 & \mbox{in }B_1\setminus \Omega_m,
\end{cases}
\end{equation}
for all $a,b\in \R$.

Now, using \eqref{difK's} we obtain
\[
\begin{split}
\frac{|K_{1,j}- K_{2^{-N},j}|}{\theta(2^{-N})}
& \le \sum_{l=0}^{N-1}  \frac{|K_{2^{-N+l+1},j}- K_{2^{-N+l},j}|}{\theta(2^{-N})}
\\
& = \sum_{l=0}^{N-1} 10\, \frac{\theta(2^{-N+l})}{\theta(2^{-N})} 2^{(-N+l)(\alpha-\epsilon)}
\\
&\le 10\, \sum_{l=0}^{N-1} 2^{(-N+l)(\alpha-\epsilon)} \le C,
\end{split}
\]
since $\alpha-\epsilon>0$.

Next, using \eqref{integral2} ---that holds with $\varphi$ replaced by $\varphi_j$---,the definition $K_{r,j}$, and that $\|\varphi_j\|_{L^\infty(B_1)}=1$ while $\|u_j\|_{L^\infty(B_1)}\le C_0$, we obtain
\begin{equation}\label{K1j}
 \bigl| K_{1,j}\bigr|  = \left| \frac{\int_{B_1} u_j\,\varphi_j\,dx}{\int_{B_1} \varphi_j^2\,dx}\right| \le C.
\end{equation}

Thus
\[  \frac{\bigl| K_{2^{-N},j}\bigr|}{\theta(2^{-N})} \le \frac{\bigl| K_{1,j}\bigr|}{\theta(2^{-N})} +\frac{\bigl| K_{1,j}- K_{2^{-N},j}\bigr|}{\theta(2^{-N})} \le C \]

Using this control for $K_{r,j}$ and setting in \eqref{eqnpointbdryreg2}
\[ a = \frac{1}{\theta(r'_m)} \quad \mbox{and} \quad b=  \frac{-K_{r'_m,j_m}}{\theta(r'_m)}\]
we obtain
\[
\begin{split}
M^+ v_m & =  \frac{(r'_m)^{2s}}{(r'_m)^{s+\alpha}\theta(r'_m)} M^+\left( \frac{1}{\theta(r'_m)} \,u_m - \frac{K_{r'_m,j_m}}{\theta(r'_m)} \varphi_m\right) (r'_m \,\cdot\,)
\\
&\ge  - C\delta\,\frac{d_m^{\alpha-s}}{\theta(r_m')}  \quad \mbox{in } B_{(r'_m)^{-1}}\cap (r'_m)^{-1}\Omega_m,
\end{split}
\]
where $d_m(x)={\rm dist}(x,r_m^{-1}\Omega_{k_m})$.
Similarly, changing sign in the previous choices of $a$ and $b$ we obtain
\[
-M^- (v_m) = M^+ (-v_m)  \ge  - C\delta\,\frac{d_m^{\alpha-s}}{\theta(r_m')}  \quad \mbox{in } B_{(r'_m)^{-1}}\cap (r'_m)^{-1}\Omega_m
\]
As complement datum we clearly have
\[v_m =  0 \quad \mbox{in }B_{(r'_m)^{-1}}\setminus (r'_m)^{-1} \Omega_m.\]

Then, by Lemma \ref{C-alpha-est} we have
\[ \|v_m\|_{C^\gamma(B_R)} \le  C(R) \quad \mbox{for all } m \mbox{ large enough}.\]
The constant $C(R)$ depends on $R$, but not on $m$.

Then, by Arzel\`a-Ascoli and the stability lemma in \cite[Lemma 4.3]{CS2} we obtain
that
\[ v_m\rightarrow v_\infty \in C(\R^n),\]
locally uniformly, where $v_\infty$ satisfies the growth control
\[ \|v_\infty\|_{L^\infty(B_R)} \le  C R^{s+\alpha}\quad \mbox{for all } R\ge 1\]
and solves \eqref{limitting} in the viscosity sense.
Thus, by the Liouville-type result \cite[Proposition 5.1]{RS-K}, we find $v_\infty(x)=K(x\cdot e)_+^s$ for some $K\in \R$.

\vspace{2mm}

{\em Step 4.} We prove that as subsequence of $\tilde\varphi_m$, where
\[ \tilde  \varphi_m (x) = \frac{\varphi_m(r'_m x)}{ \sup_{B_{r'_m}} \varphi_m},\]
converges locally uniformly to $(x\cdot e)_+^s$.

This is similar to Step 3 and we only need to use the estimates in Step 1, and the growth control \eqref{assumptiongrowth}, to obtain a uniform control of the type
\[ \|\tilde \varphi_m \|_{L^\infty(B_R)} \le C_0R^{s+\alpha} \quad \mbox{for all }R\ge1.\]
Using the estimates in Step 1  we easily show that
\[ \frac{(r'_m)^{2s}}{\sup_{B_{r'_m}} \varphi_m} \downarrow 0.\]
Thus, we use \eqref{eqnpointbdryreg2} with $a=0$ and $b= \bigl(\sup_{B_{r'_m}} \varphi_m\bigr)^{-1}$ to prove that $\tilde \varphi_m$ converges locally uniformly  to a solution $\tilde \varphi_\infty$ of
\[\begin{cases}
M^+ \tilde \varphi_\infty \ge  0 \ge M^- \tilde \varphi_\infty \quad &\mbox{in } \{e\cdot x >0\}\\
\tilde \varphi_\infty =  0 &\mbox{in } \{e\cdot x <0\},
\end{cases}\]
Then, using the Liouville-type result \cite[Proposition 5.1]{RS-K} and since
\[ \|\tilde \varphi_\infty\|_{L^\infty(B_1)} = \lim_{m\to \infty} \|\tilde \varphi_m\|_{L^\infty(B_1)}  =  \lim_{m\to \infty}1 =1\]
we get
\[ \tilde \varphi_\infty \equiv (x\cdot e)_+^s.\]
Hence, $\tilde\varphi_m(x)\rightarrow(x\cdot e)_+^s$ locally uniformly in $\R^n$.

\vspace{2mm}

{\em Step 5.}
We have $v_m\rightarrow K(x\cdot e)_+^s$ and $\tilde \varphi_m\rightarrow (x\cdot e)_+^s$ locally uniformly.
Now, by \eqref{orthog},
\[\int_{B_1} v_m(x) \tilde \varphi_m(x) \,dx=0.\]
Thus, passing this equation to the limits,
\[\int_{B_1} v_\infty (x) (x\cdot e)_+^s \,dx =0.\]
This implies $K=0$ and $v_\infty\equiv 0$.

But then passing to the limit \eqref{nondeg0} we get
\[\|v_\infty\|_{L^\infty(B_1)}\geq\frac12,\]
a contradiction.
\end{proof}

We next prove Theorems \ref{thmC1} and \ref{thmC1-2}.

\begin{proof}[Proof of Theorem \ref{thmC1-2}]
{\em Step 1.}
We first show, by a barrier argument, that for any given $\epsilon>0$ we have
\[cd^{s+\epsilon} \le u_i\leq Cd^{s-\epsilon}  \quad \mbox{in }B_{1/2},\]
where $d= {\rm dist}(\,\cdot\,,B_1\setminus\Omega)$, and $c>0$ is a constant depending only on $\Omega$, $n$, $s$, ellipticity constants.

First, notice that by assumption we have $M^- u_i = -M^+(-u_i) \le \delta$ and $M^+u_i\geq-\delta$ in $B_1\cap \Omega$.
Therefore, since $\sup_{B_{1/2}}u_i\geq1$, for any small $\rho>0$ by the interior Harnack inequality we find
\[\inf_{B_{3/4}\cap \{d\geq \rho\}}u_i\geq C^{-1}-C\delta\geq c>0,\]
provided that $\delta$ is small enough (depending on $\rho$).

Now, let $x_0\in B_{1/2}\cap \partial\Omega$, and let $e\in S^{n-1}$ be the normal vector to $\partial\Omega$ at $x_0$.
By the previous inequality,
\[\inf_{B_\rho(x_0+2\rho e)}u_i\geq c.\]
Since $\Omega$ is $C^1$, then for any $\eta>0$ there is $\rho>0$ for which
\[(x_0+\mathcal C_\eta)\cap B_{4\rho}\subset \Omega,\]
where $\mathcal C_\eta$ is the cone in Lemma \ref{homog-subsol}.

Therefore, using the function $\Phi_{\rm sub}$ given by Lemma \ref{homog-subsol}, we may build the subsolution
\[\psi=\Phi_{\rm sub}\chi_{B_{4\rho}(x_0)}+C_1\chi_{B_{\rho/2}(x_0+2\rho e)}.\]
Indeed, if $C_1$ is large enough then $\psi$ satisfies
\[M^-\psi\geq 1\quad\textrm{in}\quad (x_0+\mathcal C_\eta)\cap \bigl(B_{3\rho}(x_0)\setminus B_{\rho}(x_0+2\rho e)\bigr)\]
and $\psi\equiv0$ outside $x_0+\mathcal C_\eta$.

Hence, we may use $c_2\psi$ as a barrier, with $c_2$ small enough so that $u_i\geq c_2\psi$ in $B_{\rho}(x_0+2\rho e)$.
Then, by the comparison principle we find
\[u_i\geq c_2\psi,\]
and in particular
\[u_i(x_0+te)\geq c_3\,t^{s+\epsilon}\]
for $t\in (0,\rho)$.
Since this can be done for all $x_0\in B_{1/2}\cap \partial\Omega$, we find
\begin{equation}\label{step1}
u_i\geq c\,d^{s+\epsilon}\quad\textrm{in}\ B_{1/2}.
\end{equation}
Similarly, using the supersolution $\Phi_{\rm sup}$ from Lemma \ref{homog-subsol}, we find
\begin{equation}\label{step1B}
u_i\leq C\,d^{s-\epsilon}\quad\textrm{in}\ B_{1/2},
\end{equation}
for $i=1,2$.

\vspace{2mm}

{\em Step 2.}
Let us prove now that
\begin{equation}\label{comparable}
u_1\leq Cu_2\quad\textrm{in}\quad B_{1/2}.
\end{equation}
To prove \ref{comparable}, we rescale the functions $u_1$ and $u_2$ and use Proposition \ref{bdryregC1doms}.

Let $x_0\in B_{1/2}\cap \partial\Omega$, and let
\[\theta(r)=\sup_{r'>r}\frac{\|u_1\|_{L^\infty(B_{r'}(x_0))}+\|u_2\|_{L^\infty(B_{r'}(x_0))}}{(r')^{s+\epsilon}}.\]
Notice that $\theta(r)$ is monotone nonincreasing and that $\theta(r)\to\infty$ by \eqref{step1}.
Let $r_k\to 0$ be such that
\[\|u_1\|_{L^\infty(B_{r_k}(x_0))}+\|u_2\|_{L^\infty(B_{r_k}(x_0))}\geq \frac12\,(r_k)^{s+\epsilon}\theta(r_k),\]
with $c_0>0$, and define
\[v_k(x)=\frac{u_1(x_0+r_k x)}{(r_k)^{s+\epsilon}\theta(r_k)},\qquad w_k(x)=\frac{u_2(x_0+r_k x)}{(r_k)^{s+\epsilon}\theta(r_k)}.\]
Note that
\[\|v_k\|_{L^\infty(B_1)}+\|w_k\|_{L^\infty(B_1)}\geq \frac12.\]
Moreover,
\[\|v_k\|_{L^\infty(B_R)}=\frac{\|u_1\|_{L^\infty(B_{r_kR})}}{(r_k)^{s+\epsilon}\theta(r_k)}\leq \frac{\theta(r_kR)(r_kR)^{s+\epsilon}}{(r_k)^{s+\epsilon}\theta(r_k)}\leq R^{s+\epsilon},\]
for all $R\geq1$, and analogously
\[\|w_k\|_{L^\infty(B_R)}\leq R^{s+\epsilon}\]
for all $R\geq1$.

Now, the functions $v_k,w_k$ satisfy the equation
\[M^+(a v_k+bw_k)(x)=\frac{(r_k)^{2s}}{(r_k)^{s+\epsilon}\theta(r_k)}\,M^+(au_1+bu_2)(x_0+r_kx)\geq -C_0(r_k)^{s-\epsilon}\delta(|a|+|b|)\]
in $\Omega_k\cap B_{r_k^{-1}}$, where $\Omega_k=r_k^{-1}(\Omega-x_0)$.

Taking $k$ large enough, we will have that $\Omega_k$ satisfies the hypotheses of Proposition~\ref{bdryregC1doms} in $B_{1/\delta}$, and
\[M^+(a v_k+bw_k)\geq -\delta(|a|+|b|)\qquad \mbox{in}\quad \Omega_k\cap B_{1/\delta}.\]
Moreover, since $\sup_{B_1}v_k+\sup_{B_1}w_k\geq 1/2$, then either $\sup_{B_1}v_k\geq 1/4$ or $\sup_{B_1}w_k\geq 1/4$.
Therefore, by Proposition \ref{bdryregC1doms} we find that either
\[|v_k(x)-K_1w_k(x)|\leq C|x|^{s+\alpha}\]
or
\[|w_k(x)-K_2v_k(x)|\leq C|x|^{s+\alpha}\]
for some $|K|\leq C$.
This yields that either
\begin{equation}\label{case1}
|u_1(x)-K_1u_2(x)|\leq C|x-x_0|^{s+\alpha}
\end{equation}
or
\begin{equation}\label{case2}
|u_2(x)-K_2u_1(x)|\leq C|x-x_0|^{s+\alpha},
\end{equation}
with a bigger constant $C$.

Now, we may choose $\epsilon>0$ so that $\epsilon<\alpha/2$, and then \eqref{case2} combined with \eqref{step1}-\eqref{step1B} gives $K_2\geq c>0$, which in turn implies \eqref{case1} for $K_1=K_2^{-1}$, $|K_1|\leq C$.
Thus, in any case \eqref{case1} is proved.

In particular, for all $x_0\in B_{1/2}\cap \partial\Omega$ and all $x\in B_{1/2}\cap \Omega$ we have
\[u_1(x)/u_2(x)\leq K_1+\left|\frac{u_1(x)}{u_2(x)}-K_1\right|\leq K_1+C|x-x_0|^{s+\alpha}/u_2(x).\]
Choosing $x_0$ such that $|x-x_0|\leq Cd(x)$ and using \eqref{step1B}, we deduce
\[u_1(x)/u_2(x)\leq K_1+Cd^{s+\alpha}/d^{s-\epsilon}\leq C,\]
and thus \eqref{comparable} is proved.

\vspace{2mm}

{\em Step 3.}
We finally show that $u_1/u_2\in C^\alpha(\overline\Omega\cap B_{1/2})$ for all $\alpha\in(0,\bar\alpha)$.
Since this last step is somewhat similar to the proof of Theorem 1.2 in \cite{RS-stable}, we will omit some details.

We use that, for all $\alpha\in(0,\bar\alpha)$ and all $x\in B_{1/2}\cap \Omega$, we have
\begin{equation}\label{we-use}
\left|\frac{u_1(x)}{u_2(x)}-K(x_0)\right|\leq C|x-x_0|^{\alpha-\epsilon},
\end{equation}
where $x_0\in B_{1/2}\cap \partial\Omega$ is now the closest point to $x$ on $B_{1/2}\cap \partial\Omega$.
This follows from \eqref{case1}, as shown in Step 2.

We also need interior estimates for $u_1/u_2$.
Indeed, for any ball $B_{2r}(x)\subset\Omega\cap B_{1/2}$, with $2r=d(x)$, there is a constant $K$ such that $\|u_1-Ku_2\|_{L^\infty(B_r(x))}\leq Cr^{s+\alpha}$.
Thus, by interior estimates we find that $[u_1-Ku_2]_{C^{\alpha-\epsilon}(B_r(x))}\leq Cr^{s+\epsilon}$.
This, combined with \eqref{step1}-\eqref{step1B} yields
\begin{equation}\label{lkj}
[u_1/u_2]_{C^{\alpha-\epsilon}(B_r(x))}\leq C.
\end{equation}

Let now $x,y\in B_{1/2}\cap \Omega$, and let us show that
\begin{equation}\label{final}
\left|\frac{u_1(x)}{u_2(x)}-\frac{u_1(y)}{u_2(y)}\right|\leq C|x-y|^{\alpha-\epsilon}.
\end{equation}
If $y\in B_r(x)$, $2r=d(x)$, or if $x\in B_r(y)$, $2r=d(y)$, then this follows from \eqref{lkj}.
Otherwise, we have $|x-y|\geq \frac12\max\{d(x),d(y)\}$, and then \eqref{final} follows from \eqref{we-use}.

In any case, \eqref{final} is proved, and therefore we have
\[\|u_1/u_2\|_{C^{\alpha-\epsilon}(\overline\Omega\cap B_{1/2})}\leq C.\]
Since this can be done for any $\alpha\in(0,\bar\alpha)$ and any $\epsilon>0$, the result follows.
\end{proof}

\begin{proof}[Proof of Theorem \ref{thmC1}]
The proof is the same as Theorem \ref{thmC1-2}, replacing the Liouville-type result \cite[Proposition 5.1]{RS-K} by \cite[Theorem 4.1]{RS-stable}, and replacing $\bar\alpha$ by~$s$.
\end{proof}

\begin{rem}
Notice that in Proposition \ref{bdryregC1doms} we only require the right hand side of the equation to be bounded by $d^{\alpha-s}$.
Thanks to this, Theorem \ref{thmC1} holds as well for 
\begin{equation}\label{parabolic}
-\delta\,d^{\alpha-s} \leq f_i(x)\leq C_0d^{\alpha-s},\qquad\alpha\in(0,s).
\end{equation}
In that case, we get 
\[\|u_1/u_2\|_{C^\alpha(\Omega\cap B_{1/2})}\leq CC_0,\]
with the exponent $\alpha$ in \eqref{parabolic}.
\end{rem}

\begin{proof}[Proof of Corollary \ref{corC1}]
The result follows from Theorem \ref{thmC1}.
\end{proof}


\begin{thebibliography}{00}

\bibitem[BBB91]{BHP2} R. Ba\~nuelos, R. Bass, K. Burdzy, \emph{H\"older domains and the boundary Harnack principle}, Duke Math. J. 64 (1991) 195–200.

\bibitem[BB94]{BHP3} R. Bass, K. Burdzy, \emph{The boundary Harnack principle for non-divergence form elliptic operators}, J. Lond. Math. Soc. 50 (1994), 157-169.

\bibitem[Bog97]{Bogdan1} K. Bogdan, \emph{The boundary Harnack principle for the fractional Laplacian}, Studia Math., 123 (1997), 43-80.

\bibitem[BKK08]{Bogdan2} K. Bogdan, T. Kulczycki, and M. Kwasnicki, \emph{Estimates and structure of $\alpha$-harmonic functions}, Probab. Theory Related Fields 140 (2008), 345-381.

\bibitem[BKK15]{Bogdan3} K. Bogdan, T. Kumagai, M. Kwasnicki, \emph{Boundary Harnack inequality for Markov processes with jumps}, Trans. Amer. Math. Soc. 367 (2015), 477-517.

\bibitem[CRS15]{CRS-obstacle} L. Caffarelli, X. Ros-Oton, J. Serra, \emph{Obstacle problems for integro-differential operators: regularity of solutions and free boundaries}, preprint arXiv (Jan. 2016).

\bibitem[CS09]{CS} L. Caffarelli, L. Silvestre, \emph{Regularity theory for fully nonlinear integro-differential equations}, Comm. Pure Appl. Math. 62 (2009), 597-638.

\bibitem[CS11b]{CS2} L. Cafarelli, L. Silvestre, \emph{Regularity results for nonlocal equations by approximation}, Arch. Rat. Mech. Anal. 200 (2011), 59-88.

\bibitem[Dah77]{BHP1}  B. Dahlberg, \emph{Estimates of harmonic measure}, Arch. Rat. Mech. Anal. 65 (1977) 275–288.

\bibitem[DS14]{DS} D. De Silva, O. Savin, \emph{Boundary Harnack estimates in slit domains and applications to thin free boundary problems}, Rev. Mat. Iberoam., to appear.

\bibitem[Gru15]{Grubb} G. Grubb, \emph{Fractional Laplacians on domains, a development of H\"ormander's theory of $\mu$-transmission pseudodifferential operators}, Adv. Math. 268 (2015), 478-528.

\bibitem[Gru14]{Grubb2} G. Grubb, \emph{Local and nonlocal boundary conditions for $\mu$-transmission and fractional elliptic pseudodifferential operators}, Anal. PDE 7 (2014), 1649-1682.

\bibitem[RS14]{RS-K} X. Ros-Oton, J. Serra, \emph{Boundary regularity for fully nonlinear integro-differential equations}, Duke Math. J., to appear.

\bibitem[RS14b]{RS-stable} X. Ros-Oton, J. Serra, \emph{Regularity theory for general stable operators}, J. Differential Equations, to appear.

\bibitem[RV15]{RV} X. Ros-Oton, E. Valdinoci, \emph{The Dirichlet problem for nonlocal operators with singular kernels: convex and non-convex domains}, Adv. Math. 288 (2016), 732-790.

\bibitem[SW99]{SW} R. Song, J.-M. Wu, \emph{Boundary Harnack principle for symmetric stable processes}, J. Funct. Anal. 168 (1999), 403-427.

\end{thebibliography}
\end{document}